\documentclass[11pt]{article}
\usepackage{graphicx}
\usepackage[utf8x]{inputenc}
\usepackage{amsmath,amsthm,amssymb,fontenc,color,tensor,textcomp,amsfonts,relsize,graphicx,graphics,enumerate}
\usepackage{fancyhdr}

\usepackage{mathtools}
\usepackage{amsmath,tikz}
\usepackage{tensor}
\usepackage[rmargin=1.5cm, lmargin=1.5cm]{geometry}
\newtheorem{theo}{Theorem}[section]
\newtheorem{lem}[theo]{Lemma}

\newtheorem{prop}[theo]{Proposition}

\newtheorem{cl}[theo]{Claim}

\usepackage[colorlinks=true,
            linkcolor=red,
            urlcolor=blue,
            citecolor=green]{hyperref}

\usepackage{url}
\makeatletter
\newcommand{\address}[1]{\gdef\@address{#1}}
\newcommand{\email}[1]{\gdef\@email{\url{#1}}}
\newcommand{\@endstuff}{\par\vspace{\baselineskip}\noindent\small
\begin{tabular}{@{}l}\scshape\@address\\\textit{E-mail address:} \@email\end{tabular}}
\AtEndDocument{\@endstuff}
\makeatother
\title{Rigidity Result for Four Dimensional Willmore Submanifolds with Boundary}
\author{Peter Olamide Olanipekun}
\address{Department of Mathematics, The University of Auckland, Auckland 1010, New Zealand.}
\email{peter.olanipekun@auckland.ac.nz }

\date{}
\begin{document}\sloppy
\maketitle
\abstract{We establish a rigidity result for the critical points, with boundary, of a four dimensional Willmore energy. These critical points satisfy a 4-Willmore equation which is a sixth order nonlinear elliptic partial differential equation. We establish several curvature estimates and prove that four dimensional Willmore submanifold with totally geodesic boundary condition are umbilic. }
\section{Introduction}
The goal of this present paper is to establish a rigidity result for the critical points of the four dimensional Willmore energy.   This four dimensional energy is relatively new in literature and has only been studied from analytical perspective in \cite{peterthesis} where a regularity result was presented. We believe that our results will encourage further studies of the four dimensional energy from both geometric and analytical perspectives (and even computational perspective), particularly, higher order flow of four dimensional Willmore submanifolds, stability of four dimensional Willmore submanifolds and energy quantization results amongst others. 

\vskip3mm
\noindent
Let $\vec{\Phi}:\Sigma\rightarrow\mathbb{R}^m$ be an immersion of a four dimensional manifold $\Sigma$.
Denote by $g$, $\vec{h}$ and $\vec{H}$ the first fundamental form, second fundamental form and the mean curvature vector, respectively, of $\Sigma$. The four dimensional Willmore energy is given by
\begin{align}
\mathcal{E}(\Sigma):= \int_\Sigma (|\pi_{\vec{n}}\nabla\vec{H}|^2 -|\vec{H}\cdot\vec{h}|^2 +7|\vec{H}|^4 ) \,d\mu
\end{align}
where $d\mu$ is the volume element. This energy\footnote{See Section 1.2.1 of \cite{peterthesis} and the references therein for a more detailed discussion.}  was first derived by Jemal Guven in \cite{guven}. It was later rediscovered by Robin-Graham and Reichert in \cite{robingraham} and also by Zhang in \cite{zhang} where higher Willmore energies were constructed via minimal submanifold asymptotics.  Examples of critical points of $\mathcal{E}(\Sigma)$ include minimal submanifolds, totally geodesic submanifolds of $\mathbb{R}^m$, round spheres $\mathbb{S}^4$ of $\mathbb{R}^m$ and product of spheres of radius $r$ such as: $\mathbb{S}^2 (r_1) \times \mathbb{S}^2 (r_2) $, $\mathbb{S}^1 (r_1) \times \mathbb{S}^3 (r_2) $, $\mathbb{S}^1 (r_1) \times \mathbb{S}^1 (r_2)\times \mathbb{S}^2 (r_3)   $ and $\mathbb{S}^1 (r_1) \times \mathbb{S}^1 (r_2) \times \mathbb{S}^1 (r_3) \times \mathbb{S}^1 (r_4)  $ where $\sum r^2_i =1$. These are solutions of the 4-Willmore equation which will be derived in Section \ref{will67} as the first variation of the energy $\mathcal{E}(\Sigma)$.

\vskip2mm
\noindent
The rigidity of several kinds of submanifolds has been widely studied in literature under different contexts. For instance, while some rigidity results for manifolds with bounded Ricci curvature were obtained  in \cite{anderson} other studies have focused on minimal submanifolds \cite{gqchen, chern, dorel,  coibre, lawson, reilly, jsimons}, critical points of the Willmore functional \cite {kuwert, tobias, tobiasl} and  hypersurfaces of constant weighted mean curvature \cite{ancari, cheng, shu}. 
In \cite{mmcoy}, McCoy and Wheeler considered surfaces $\Sigma$ immersed into $\mathbb{R}^3$ which are critical points of the  functional
$$\int_\Sigma |\nabla H|^2\, d\mu$$
and  whose second fundamental form satisfies the smallness condition 
$$\int_\Sigma |h|^2\, d\mu \leq \varepsilon$$
where $\varepsilon$ is a small universal constant. They obtained the following result.
\begin{theo}
Let $f:\Sigma \rightarrow\mathbb{R}^3$ be an immersion satisfying
$$\Delta^2H +|h|^2 \Delta H- (h_0)^{ij} \nabla_iH \nabla_j H=0$$
with the boundary conditions
$$|h|=0 \quad\mbox{and}\quad \nabla_\eta H= \nabla_\eta \Delta H=0.$$
If $f$ also satisfies $\int_\Sigma |h|^2 d\mu \leq \varepsilon$ for some sufficiently small $\varepsilon >0$, then the immersed surface $f(\Sigma) $ is part of a flat plane, where $\eta$ is the unit  conormal to the boundary of $\Sigma$.

\end{theo}

\vskip3mm
Our main result is the following rigidity theorem for critical points of the energy $\mathcal{E}(\Sigma)$. 
\begin{theo} \label{mainr}
Let $\vec{\Phi}:\Sigma\rightarrow \mathbb{R}^m$ be an immersion of a 4-dimensional manifold $\Sigma$ satisfying $\int_\Sigma |\vec{h}|^2 d\mu\leq \varepsilon$ and $\int_\Sigma |\vec{h}|^4 d\mu\leq \varepsilon$ for some sufficiently small $\varepsilon >0$. If $\vec{\Phi}$ also satisfies
\begin{align}
\vec{\mathcal{W}}&=\vec{0} \label{654t}
\end{align}
together with the boundary conditions 
\begin{align}
\pi_{\vec{\eta}}\nabla\Delta_\perp \vec{H}=\pi_{\vec{\eta}}\nabla \vec{H}=\vec{0} \quad\mbox{and}\quad \vec{h}=\vec{0}    \label{bdry}
\end{align}
where
\begin{align}
\vec{\mathcal{W}}&:=-\frac{1}{2}\Delta^2_{\perp}\vec{H}-\frac{1}{2}(\vec{h}_{ik}\cdot\Delta_{\perp}\vec{H})\vec{h}^{ik}-4|\pi_{\vec{n}}\nabla\vec{H}|^2\vec{H} +2\pi_{\vec{n}}\nabla_j((\vec{h}^j_i\cdot\nabla^i\vec{H})\vec{H})
\nonumber
\\& \quad
-2\pi_{\vec{n}}\nabla_j((\vec{H}\cdot\vec{h}^j_i)\pi_{\vec{n}}\nabla^i\vec{H})
+2(\pi_{\vec{n}}\nabla_i\vec{H}\cdot\pi_{\vec{n}}\nabla_j\vec{H})\vec{h}^{ij}
\nonumber
                                \\&\quad-\frac{1}{2}\Delta_{\perp}((\vec{H}\cdot\vec{h}^{ij})\vec{h}_{ij})-2\pi_{\vec{n}}\nabla_i\nabla_k((\vec{H}\cdot\vec{h}^{ik})\vec{H}) -28|\vec{H}|^4\vec{H}                                                 \nonumber
\\&\quad      -\frac{1}{2}(\vec{H}\cdot\vec{h}^{ij})(\vec{h}_{ij}\cdot \vec{h}_{pq})\vec{h}^{pq}     -4(\vec{H}\cdot\vec{h}_{ij})(\vec{H}\cdot\vec{h}^i_k)\vec{h}^{jk}  +4|\vec{H}\cdot\vec{h}|^2\vec{H}  \nonumber
\\&\quad   +7 \Delta_\perp(|\vec{H}|^2\vec{H})  +7|\vec{H}|^2(\vec{H}\cdot\vec{h}_{ij})\vec{h}^{ij} \nonumber 
\end{align}
then the submanifold $\Sigma$ is umbilic with totally geodesic boundary. 
\end{theo}
\noindent
Note that the Willmore equation \eqref{654t} is a sixth order nonlinear partial differential equation.
\noindent
The remaining part of this paper is organised as follows. Notations and brief details are given Section \ref{two84673}.  Section \ref{will67} is devoted to obtaining the 4-Willmore equation by computing the normal variation of $\mathcal{E}(\Sigma)$. In Section \ref{three0009}, we obtain several curvature estimates needed to prove the main result of this article.

\section{Preliminaries}\label{two84673}

Let $\vec{\Phi}: \Sigma \rightarrow \mathbb{R}^m$ be a smooth immersion of a manifold $\Sigma$ into $\mathbb{R}^m$. The induced metric $g$ on $\Sigma$ is given by 
$$g_{ij}:= \nabla_i\vec{\Phi}\cdot\nabla_j\vec{\Phi}$$
where $\nabla_{(.)}$ is a covariant derivative associate with $g$ and $g^{ij}$ is the inverse metric. Note that $g^{ij}g_{ij}=4=\textnormal{dim}(\Sigma)$. We append arrows on vector quantities and use a dot to   denote the scalar product of vectors in $\mathbb{R}^m$. The volume element on $\Sigma$ will be denoted by $$d\mu:= dvol_g= |g|^{1/2} dx$$ and on the boundary $\partial \Sigma$ by $d\sigma$. The components of the second fundamental form of $\Sigma$
 are given  by
$$\vec{h}_{ij}:=\nabla_i\nabla_j\vec{\Phi}$$
from which we find the mean curvature vector $4\vec{H}=g^{ij}\vec{h}_{ij}=\vec{h}_j^j$.
We denote by $\pi_{\vec{n}}\vec{X}$ and $\pi_T\vec{X}$ the normal and tangential projections of a vector $\vec{X}$ respectively and we have $\pi_{\vec{n}}\vec{X}+\pi_T\vec{X}=\vec{X}$. We will often decompose vectors into normal and tangential components.
For a vector $\vec{X}$, we denote by 
$$\Delta_\perp \vec{X}:=g^{ij}\pi_{\vec{n}}\nabla_i\pi_{\vec{n}}\nabla_j\vec{X}$$  the negative covariant Laplacian for the Levi-Civita connection in the normal bundle. We also denote 
$$\Delta_\perp^2 \vec{X}:=g^{kl}g^{ij}\pi_{\vec{n}}\nabla_k\pi_{\vec{n}}\nabla_l\pi_{\vec{n}}\nabla_i\pi_{\vec{n}}\nabla_j\vec{X}.$$

\noindent
The tracefree second fundamental form of $\Sigma$ is the symmetric tensor $\vec{h}_0$ whose components are given by
$$(\vec{h}_0)_{ij}=\vec{h}_{ij}-g_{ij}\vec{H}.$$

\noindent
The tracefree second fundamental form  satisfies
$$\pi_{\vec{n}}\nabla^i(\vec{h}_0)_{ij}=\pi_{\vec{n}}\nabla^i\vec{h}_{ij}-\pi_{\vec{n}}\nabla_j\vec{H} = 3\pi_{\vec{n}}\nabla_j\vec{H}$$
where we have used the Codazzi-Mainardi equations
$$\pi_{\vec{n}}\nabla_i\vec{h}_{jk}=\pi_{\vec{n}}\nabla_j\vec{h}_{ik}= \pi_{\vec{n}}\nabla_k\vec{h}_{ij}.$$

\noindent
We now introduce the cut-off function that will be needed whenever we are integrating by parts.  Let $\vec{\Phi}:\Sigma\rightarrow \mathbb{R}^m$ be our immersion and let $\tilde \gamma:\mathbb{R}^m\rightarrow [0,1]$ be a $C^1$ function with compact support, we take the composition 
\begin{align}
\gamma= \tilde \gamma \circ \vec{\Phi}: \Sigma \rightarrow [0,1] \label{cutoff}
\end{align}
with the assumption that $\|\nabla\gamma\|_{L^\infty}\leq c_\gamma$. We set $c_\gamma=\frac{c}{\rho}$ where $c$ is an absolute constant and $\rho$ will be specified later.

\section{The 4-Willmore equation} \label{will67}
We derive the normal variation of the energy $\mathcal{E}(\Sigma)$.
\begin{theo}
Let $\vec{\Phi}:\Sigma\rightarrow \mathbb{R}^m$ be an immersion of a 4-dimensional manifold $\Sigma$. Let $\vec{B}:\Sigma \rightarrow \mathbb{R}^m$ be a smooth normal variation of $\vec{\Phi}$. The first variation of the energy is given by
\begin{align}
\tilde\delta \mathcal{E}(\Sigma)=     \int_\Sigma \vec{B}\cdot \vec{\mathcal{W}}\,\, d\mu + \int_{\partial \Sigma} V\, d\sigma
\end{align}
where
\begin{align}
\vec{\mathcal{W}}&:=-\frac{1}{2}\Delta^2_{\perp}\vec{H}-\frac{1}{2}(\vec{h}_{ik}\cdot\Delta_{\perp}\vec{H})\vec{h}^{ik}-4|\pi_{\vec{n}}\nabla\vec{H}|^2\vec{H} +2\pi_{\vec{n}}\nabla_j((\vec{h}^j_i\cdot\nabla^i\vec{H})\vec{H})
\nonumber
\\& \quad
-2\pi_{\vec{n}}\nabla_j((\vec{H}\cdot\vec{h}^j_i)\pi_{\vec{n}}\nabla^i\vec{H})
+2(\pi_{\vec{n}}\nabla_i\vec{H}\cdot\pi_{\vec{n}}\nabla_j\vec{H})\vec{h}^{ij}
\nonumber
                                \\&\quad-\frac{1}{2}\Delta_{\perp}((\vec{H}\cdot\vec{h}^{ij})\vec{h}_{ij})-2\pi_{\vec{n}}\nabla_i\nabla_k((\vec{H}\cdot\vec{h}^{ik})\vec{H}) -28|\vec{H}|^4\vec{H}                                                 \nonumber
\\&\quad      -\frac{1}{2}(\vec{H}\cdot\vec{h}^{ij})(\vec{h}_{ij}\cdot \vec{h}_{pq})\vec{h}^{pq}     -4(\vec{H}\cdot\vec{h}_{ij})(\vec{H}\cdot\vec{h}^i_k)\vec{h}^{jk}  +4|\vec{H}\cdot\vec{h}|^2\vec{H}  \nonumber
\\&\quad   +7 \Delta_\perp(|\vec{H}|^2\vec{H})  +7|\vec{H}|^2(\vec{H}\cdot\vec{h}_{ij})\vec{h}^{ij} \nonumber
\end{align}
and 
\begin{align}
V^j&=  \frac{1}{2} (\nabla^j\vec{H}\cdot \vec{h}^{ik})(\vec{B}\cdot\vec{h}_{ik}) +\frac{1}{2} \nabla^j\vec{H}\cdot\Delta_\perp \vec{B} -2(\nabla^i\vec{H}\cdot\vec{h}_i^j)(\vec{H}\cdot\vec{B})                   \nonumber
\\&\quad+2(\nabla^i\vec{H}\cdot\vec{B})(\vec{H}\cdot\vec{h}_i^j)   -\frac{1}{2} \Delta_\perp\vec{H}\cdot\nabla^j\vec{B} +\frac{1}{2}\nabla^j\Delta_\perp\vec{H}\cdot\vec{B}   \nonumber
\\&
-\frac{1}{2}(\vec{H}\cdot\vec{h}^{ik})(\vec{h}_{ik}\cdot\nabla^j\vec{B})             +\frac{1}{2}\pi_{\vec{n}}\nabla^j((\vec{H}\cdot\vec{h}^{ik})\vec{h}_{ik})\cdot\vec{B} -2(\vec{H}\cdot\vec{h}^{ij})(\vec{H}\cdot\nabla_i\vec{B})     \nonumber
\\&\quad +2\pi_{\vec{n}}\nabla_i((\vec{H}\cdot\vec{h}^{ij})\vec{H})\cdot\vec{B} 
+7|\vec{H}|^2\vec{H}\cdot\nabla^j\vec{B} -7\pi_{\vec{n}}\nabla^j(|\vec{H}|^2\vec{H})\cdot\vec{B}. \nonumber
\end{align}

\end{theo}

\begin{proof}

We begin with the following straightforward variations  (c.f \cite{peterthesis}). We have
\begin{align}
&\tilde\delta g_{ij}=  -2\vec{B}\cdot\vec{h}_{ij} \,,
 \quad\quad \tilde\delta g^{ij}=2\vec{B}\cdot\vec{h}^{ij} \,,
\quad\quad  \tilde\delta|g|^{1/2}= -4|g|^{1/2}\vec{B}\cdot\vec{H}\,,   \label{evol1}
\\
&\pi_{\vec{n}}\tilde\delta \vec{h}_{ij}= \pi_{\vec{n}} \nabla_i\nabla_j \vec{B} \,, 
\quad\quad\mbox{and}\quad\pi_{\vec{n}}\tilde\delta\vec{H}=\frac{1}{4}((\vec{B}\cdot\vec{h}^{ij})\vec{h}_{ij}+\Delta_{\perp}\vec{B}).    \label{evol2}
\end{align}

\noindent
We compute the variation of each term of $\mathcal{E}(\Sigma)$. 
 \begin{align}
&\tilde\delta\int_\Sigma|\pi_{\vec{n}}\nabla\vec{H}|^2 d\mu = \tilde\delta\int_\Sigma g^{ij}\pi_{\vec{n}}\nabla_i\vec{H} \cdot\pi_{\vec{n}}\nabla_j\vec{H} \,|g|^{1/2}\, dx         \nonumber
\\&=2\int_{\Sigma} (\vec{B}\cdot\vec{h}^{ij}) \pi_{\vec{n}}\nabla_i\vec{H} \cdot\pi_{\vec{n}}\nabla_j\vec{H} \,d\mu +2\int_\Sigma \pi_{\vec{n}}\nabla^j\vec{H} \cdot\tilde\delta\pi_{\vec{n}}\nabla_j\vec{H} \,d\mu -4\int_{\Sigma} |\pi_{\vec{n}}\nabla\vec{H}|^2(\vec{B}\cdot\vec{H})\,  d\mu    \label{middle1}
\end{align}
\noindent
By splitting into tangential and normal components, the middle term in \eqref{middle1} can be simplified as follows.
\begin{align}
\pi_{\vec{n}}\nabla^j\vec{H} \cdot\tilde\delta\pi_{\vec{n}}\nabla_j\vec{H} &=  \pi_{\vec{n}}\nabla^j\vec{H} \cdot\tilde\delta\nabla_j\vec{H} -\pi_{\vec{n}}\nabla^j\vec{H} \cdot\tilde\delta\pi_{T}\nabla_j\vec{H}\nonumber
\\&= \pi_{\vec{n}}\nabla^j\vec{H} \cdot\nabla_j\tilde\delta\vec{H} -   \pi_{\vec{n}}\nabla^j\vec{H} \cdot\tilde\delta((\nabla_j\vec{H}\cdot \nabla_k\vec{\Phi}) \nabla^k\vec{\Phi} )       \nonumber
\\&=  \pi_{\vec{n}}\nabla^j\vec{H} \cdot\nabla_j\tilde\delta\vec{H} + \pi_{\vec{n}}\nabla^j\vec{H} \cdot\tilde\delta((\vec{H}\cdot \vec{h}_{jk}\vec{\Phi}) \nabla^k\vec{\Phi} )       \nonumber
\\&=   \pi_{\vec{n}}\nabla^j\vec{H} \cdot\nabla_j(\pi_{\vec{n}}\tilde\delta\vec{H} +\pi_{T}\tilde\delta\vec{H}) + \pi_{\vec{n}}\nabla^j\vec{H} \cdot((\vec{H}\cdot \vec{h}_{jk}\vec{\Phi}) \tilde\delta\nabla^k\vec{\Phi} )     .        \label{sothatt}
\end{align}
\noindent
Note that
\begin{align}
\pi_{\vec{n}}\nabla_j\pi_{T} \tilde\delta\vec{H}& =   \pi_{\vec{n}}\nabla_j((\tilde\delta\vec{H}\cdot\nabla_k\vec{\Phi}) \nabla^k\vec{\Phi})
=   -  \pi_{\vec{n}}\nabla_j((\vec{H}\cdot\tilde\delta\nabla_k\vec{\Phi}) \nabla^k\vec{\Phi}) =-\pi_{\vec{n}}\nabla_j((\vec{H}\cdot\nabla_k\vec{B}) \nabla^k\vec{\Phi})  \nonumber
\\&= -(\vec{H}\cdot\nabla_k\vec{B}) \vec{h}_{j}^k   . \nonumber
 \end{align}

\noindent
Calling upon \eqref{evol2} and the latter, equation \eqref{sothatt} becomes
\begin{align}
\pi_{\vec{n}}\nabla^j\vec{H} \cdot\tilde\delta\pi_{\vec{n}}\nabla_j\vec{H} &=  \frac{1}{4} \pi_{\vec{n}} \nabla^j\vec{H} \cdot\nabla_j\left( (\vec{B}\cdot\vec{h}^{ik})\vec{h}_{ik} +\Delta_\perp\vec{B}   \right)      \nonumber
\\& -(\pi_{\vec{n}}\nabla^j\vec{H} \cdot\vec{h}_j^k) (\vec{H}\cdot\nabla_k\vec{B})  +(\pi_{\vec{n}}\nabla^j\vec{H}\cdot\nabla_k\vec{B}) (\vec{H}\cdot\vec{h}_j^k):= \textnormal{I}+\textnormal{II} +\textnormal{III}.\label{sub1}
\end{align}
\noindent
Now we have
\begin{align}
\textnormal{I}&=   \frac{1}{4}\nabla_j \left( \pi_{\vec{n}} \nabla^j\vec{H}\cdot\vec{h}_{ik}(\vec{B}\cdot\vec{h}^{ik})  +\pi_{\vec{n}}\nabla^j\vec{H}\cdot\Delta_\perp\vec{B}  \right)      
- \frac{1}{4}(\Delta_\perp\vec{H}\cdot\vec{h}_{ik}) (\vec{B}\cdot\vec{h}^{ik})  -\frac{1}{4} \Delta_\perp\vec{H}\cdot \Delta_\perp \vec{B}. \label{sub2}
\end{align}

\noindent
Note that
\begin{align}
\Delta_\perp\vec{H}\cdot \Delta_\perp \vec{B} &= \Delta_\perp\vec{H}\cdot \pi_{\vec{n}}\nabla_i \pi_{\vec{n}}\nabla^i \vec{B}        \nonumber
= \nabla_i(\Delta_\perp\vec{H} \cdot\nabla^i\vec{B}) -\nabla_i\Delta_\perp\vec{H} \cdot\pi_{\vec{n}}\nabla^i\vec{B}   \nonumber
\\&=  \nabla_i(\Delta_\perp\vec{H} \cdot\nabla^i\vec{B}-\pi_{\vec{n}}\nabla_i\Delta_\perp \vec{H}\cdot \vec{B)} +\vec{B}\cdot\Delta_\perp^2\vec{H}   .       \label{sub3}
\end{align}

\noindent
Similar computation gives
\begin{align}
\textnormal{II} +\textnormal{III}& = \nabla_k\left(-(\nabla^j\vec{H}\cdot\vec{h}^k_j)(\vec{H}\cdot\vec{B} ) +(\nabla^j\vec{H}\cdot\vec{B}) (\vec{H}\cdot\vec{h}^k_j) \right)   \nonumber
\\&\quad + \vec{B}\cdot\nabla_k\left((\nabla^j\vec{H}\cdot \vec{h}^k_j)\vec{H}  -(\vec{H}\cdot\vec{h}^k_j)\pi_{\vec{n}}\nabla^j\vec{H}\right)    \label{sub4}
\end{align}

\noindent
Substituting \eqref{sub3} into \eqref{sub2} and using \eqref{sub1} in \eqref{middle1}, we have
\begin{align}
\tilde\delta\int_\Sigma|\pi_{\vec{n}}\nabla\vec{H}|^2 d\mu &= \int_\Sigma \vec{B}\cdot \mathcal{\vec{W}}_1 + \nabla_j V_1^j d\mu \nonumber
\end{align}
where
\begin{align}
\mathcal{\vec{W}}_1&:= -\frac{1}{2} \Delta_\perp \vec{H}\cdot\vec{h}_{ik} \vec{h}^{ik} +2\pi_{\vec{n}} \nabla_k((\nabla^j\vec{H}\cdot\vec{h}_j^k)\vec{H}) -2\pi_{\vec{n}}\nabla_k((\vec{H}\cdot \vec{h}^k_j)\pi_{\vec{n}}\nabla^j\vec{H})   \nonumber
\\& \quad-\frac{1}{2} \Delta_\perp ^2\vec{H} +2(\pi_{\vec{n}}\nabla_i\vec{H}\cdot\pi_{\vec{n}}\nabla_j\vec{H})\vec{h}^{ij}
-4|\pi_{\vec{n}}\nabla\vec{H}|^2\vec{H} \nonumber
\end{align}
and
\begin{align}
V_1^j&:= \frac{1}{2} (\nabla^j\vec{H}\cdot \vec{h}^{ik})(\vec{B}\cdot\vec{h}_{ik}) +\frac{1}{2} \nabla^j\vec{H}\cdot\Delta_\perp \vec{B} -2(\nabla^i\vec{H}\cdot\vec{h}_i^j)(\vec{H}\cdot\vec{B})                   \nonumber
\\&\quad +2(\nabla^i\vec{H}\cdot\vec{B})(\vec{H}\cdot\vec{h}_i^j) -\frac{1}{2} \Delta_\perp\vec{H}\cdot\nabla^j\vec{B} +\frac{1}{2}\nabla^j\Delta_\perp\vec{H}\cdot\vec{B}  .  \nonumber
\end{align}

\noindent
Next, we compute the variation of the second term of the energy $\mathcal{E}(\Sigma)$.
\begin{align}
\tilde\delta \int_\Sigma |\vec{H}\cdot\vec{h}|^2 \,\,d\textnormal {vol}_g &=\tilde\delta\int_{\Sigma} g^{ik} g^{jl} (\vec{H}\cdot\vec{h}_{ij}) (\vec{H}\cdot\vec{h}_{kl}) \,\,|g|^{1/2} dx  \nonumber
\\&= 2\int_\Sigma (\vec{H}\cdot\vec{h}^{ij} )\tilde\delta(\vec{H}\cdot\vec{h}_{ij}) d\mu + (\vec{H}\cdot\vec{h}_{ij} )(\vec{H}\cdot\vec{h}^i_k) \tilde\delta g^{jk}\,\, d\mu +\int_{\Sigma} |\vec{H}\cdot\vec{h}|^2\tilde\delta|g|^{1/2} dx  \nonumber
\\&  = 2\int_\Sigma (\vec{H}\cdot\vec{h}^{ij} )\tilde\delta(\vec{H}\cdot\vec{h}_{ij})\,\, d\mu  +4\int_{\Sigma} \vec{B}\cdot\left[(\vec{H}\cdot\vec{h}_{ij})(\vec{H}\cdot\vec{h}^i_k)\vec{h}^{jk} -|\vec{H}\cdot\vec{h}|^2\vec{H}   \right] d\mu   \label{firstlyy}
\end{align}

\noindent
Inspecting the first integrand of \eqref{firstlyy} we have
\begin{align}
(\vec{H}\cdot\vec{h}^{ij} )\tilde\delta(\vec{H}\cdot\vec{h}_{ij})&=  (\vec{H}\cdot\vec{h}^{ij} )(\tilde\delta\vec{H}\cdot\vec{h}_{ij} +\vec{H}\cdot\tilde\delta\vec{h}_{ij})  \nonumber
\\&= \frac{1}{4}(\vec{H}\cdot\vec{h}^{ij} )(\vec{h}_{kl}\cdot\vec{h}_{ij})(\vec{B}\cdot\vec{h}^{kl}) +\frac{1}{4}(\vec{H}\cdot\vec{h}^{ij})(\vec{h}_{ij}\cdot\Delta_\perp\vec{B}) +(\vec{H}\cdot\vec{h}^{ij})(\vec{H}\cdot\nabla_i\nabla_j\vec{B}).   \label{seco}
\end{align}

\noindent
Note that the last two terms of \eqref{seco} can be written as
\begin{align}
(\vec{H}\cdot\vec{h}^{ij})(\vec{h}_{ij}\cdot\Delta_\perp\vec{B}) &= 
\nabla_k\left[ (\vec{H}\cdot\vec{h}^{ij})\vec{h}_{ij}\nabla^k\vec{B} -\pi_{\vec{n}}\nabla^k((\vec{H}\cdot\vec{h}^{ij})\vec{h}_{ij})\cdot\vec{B} \right] + \Delta_\perp((\vec{H}\cdot\vec{h}^{ij})\vec{h}_{ij})\cdot\vec{B}\nonumber
\end{align}
and
\begin{align}
(\vec{H}\cdot\vec{h}^{ij})(\vec{H}\cdot\nabla_i\nabla_j\vec{B}) &= 
\nabla_i\left[ (\vec{h}\cdot\vec{h}^{ij})(\vec{H}\cdot\nabla_j\vec{B}) -\nabla_j((\vec{H}\cdot \vec{h}^{ij})\vec{H})\cdot\vec{B} \right] +\vec{B}\cdot\nabla_i\nabla_j((\vec{h}\cdot\vec{h}^{ij})\vec{H}) . \nonumber
\end{align}

\noindent
Substituting \eqref{seco} into \eqref{firstlyy} yields
\begin{align}
\tilde\delta \int_\Sigma |\vec{H}\cdot\vec{h}|^2 \,\,d\mu &= \int_\Sigma \vec{B}\cdot\mathcal{\vec{W}}_2 +\nabla_jV_2^j \,\,\,d\mu
\end{align}

\noindent
where
\begin{align}
\mathcal{\vec{W}}_2 =  4(\vec{H}\cdot\vec{h}_{ij})(\vec{H}\cdot\vec{h}^i_k)\vec{h}^{jk} -4 |\vec{H}\cdot\vec{h}|^2\vec{H}   +\frac{1}{2}(\vec{H}\cdot\vec{h}^{ij})(\vec{h}_{kl}\cdot\vec{h}_{ij})\vec{h}^{kl}   +\frac{1}{2}\Delta_\perp ((\vec{H}\cdot\vec{h}^{ij})\vec{h}_{ij})     +2\pi_{\vec{n}}\nabla_i\nabla_j ((\vec{H}\cdot\vec{h}^{ij})\vec{H})\nonumber
\end{align}

\noindent
and
\begin{align}
V^j_2=\frac{1}{2}(\vec{H}\cdot\vec{h}^{ik})(\vec{h}_{ik}\cdot\nabla^j\vec{B})             -\frac{1}{2}\pi_{\vec{n}}\nabla^j((\vec{H}\cdot\vec{h}^{ik})\vec{h}_{ik})\cdot\vec{B} +2(\vec{H}\cdot\vec{h}^{ij})(\vec{H}\cdot\nabla_i\vec{B})    -2\pi_{\vec{n}}\nabla_i((\vec{H}\cdot\vec{h}^{ij})\vec{H})\cdot\vec{B} .\nonumber
\end{align}

\noindent
Finally, we compute the variation of the third term of $\mathcal{E}(\Sigma)$.
\begin{align}
\tilde\delta\int_{\Sigma}|\vec{H}|^4 d\textnormal {vol}_g &= 4\int_{\Sigma} |\vec{H}|^2\vec{H}\cdot\tilde\delta\vec{H} \,\,d\mu +\int_{\Sigma}|\vec{H}|^4 \tilde\delta|g|^{1/2}dx \nonumber
\\&= \int_{\Sigma} |\vec{H}|^2\vec{H}\cdot((\vec{B}\cdot\vec{h}^{ij}) \vec{h}_{ij} +\Delta_\perp\vec{ B}) -4|\vec{H}|^4(\vec{B}\cdot\vec{H}) \,\,\, d\mu  .           \label{third}
\end{align}

\noindent
Note that
\begin{align}
|\vec{H}|^2\vec{H}\cdot\Delta_\perp\vec{B} &=  \nabla_j\left[ |\vec{H}|^2\vec{H}\cdot\nabla^j\vec{B} -\pi_{\vec{n}}\nabla^j(|\vec{H}|^2\vec{H})\cdot\vec{B} \right]   +\Delta_\perp(|\vec{H}|^2\vec{H})\cdot\vec{B} . \label{lasteq}
\end{align}

\noindent
Substituting \eqref{lasteq} into \eqref{third}, we find
\begin{align}
\tilde\delta \int_\Sigma |\vec{H}|^4 \,\,d\textnormal {vol}_g &= \int_\Sigma \vec{B}\cdot\mathcal{\vec{W}}_3 +\nabla_jV_3^j \,\,\,d\mu
\end{align}

\noindent
where
\begin{align}
\mathcal{\vec{W}}_3= |\vec{H}|^2(\vec{H}\cdot\vec{h}_{ij})\vec{h}^{ij} -4|\vec{H}|^4\vec{H} +\Delta_\perp(|\vec{H}|^2\vec{H})           \nonumber
\end{align}
and
\begin{align}
V_3^j= |\vec{H}|^2\vec{H}\cdot\nabla^j\vec{B} -\pi_{\vec{n}}\nabla^j(|\vec{H}|^2\vec{H})\cdot\vec{B}.
\end{align}

\noindent
By setting $\mathcal{\vec{W}}:= \mathcal{\vec{W}}_1-\mathcal{\vec{W}}_2+7\mathcal{\vec{W}}_3$ and $V=V_1-V_2+7V_3$, and by applying divergence theorem, we find
\begin{align}
 \tilde\delta \int_{\Sigma}|\pi_{\vec{n}}\nabla\vec{H}|^2 -|\vec{H}\cdot\vec{h}|^2 +7|\vec{H}|^4 \,\, d\mu&= \int_{\Sigma}\vec{B}\cdot\mathcal{\vec{W}} +\nabla_j V^j \,\, d\mu \nonumber
\\&=   \int_\Sigma \vec{B}\cdot \mathcal{\vec{W}}\,\,d\mu +\int_{\partial \Sigma}V \,\,d\sigma.
\end{align}

\end{proof}
\noindent
One verifies via integration by parts and Codazzi-Mainardi equations that the boundary conditions \eqref{bdry}
ensure that the boundary terms vanish. Note that one would need to apply integration by parts twice to some terms of $V$ on the boundary.  It is known that the critical points of $\mathcal{E}(\Sigma)$ corresponds to $\mathcal{\vec{W}}=\vec{0}$ (see \cite{robingraham}, \cite{peterthesis}).



\noindent
\begin{prop}\label{mekl}
Let $\vec{\Phi}:\Sigma\rightarrow \mathbb{R}^m$ be an immersion of a 4-dimensional manifold $\Sigma$ whose boundaries satisfy the condition \eqref{bdry} and let $\gamma$ be the cut-off function in \eqref{cutoff} then 
the Willmore invariant satisfies the integral identity
\begin{align}
\int_\Sigma\mathcal{\vec{W}}\cdot\vec{H} \gamma^p d\mu &= \int_\Sigma\mathcal{\vec{W}}_1\cdot\vec{H} \gamma^p d\mu -\int_\Sigma\mathcal{\vec{W}}_2\cdot\vec{H} \gamma^p d\mu +7\int_\Sigma\mathcal{\vec{W}}_3\cdot\vec{H} \gamma^p d\mu \nonumber
\\&= -\frac{1}{2}\int_{\Sigma} |\Delta_\perp\vec{H}|^2 \gamma^p d\mu               +\int_\Sigma \pi_{\vec{n}}\nabla^j\vec{H}\cdot\nabla_j[(\vec{H}\cdot\vec{h}^{ik})\vec{h}_{ik}]
\gamma^p d\mu -2\int_\Sigma (\nabla^j\vec{H}\cdot\vec{h}^k_j)(\vec{H}\cdot\nabla_k\vec{H})\gamma^p d\mu  \nonumber
\\&\quad +2 \int_\Sigma (\vec{H}\cdot\vec{h}^k_j)(\pi_{\vec{n}}\nabla^j\vec{H}\cdot\nabla_k\vec{H})\gamma^p d\mu             +2\int_\Sigma\nabla_j((\vec{H}\cdot\vec{h}^{ij})\vec{H})\cdot\nabla_i\vec{H}
\gamma^p d\mu \nonumber 
\\&\quad   -7\int_\Sigma \nabla^j(|\vec{H}|^2\vec{H})\cdot\pi_{\vec{n}}\nabla_j\vec{H} \gamma^p d\mu                 +\int_\Sigma \vec{T}\cdot\vec{H} \gamma^p d\mu  +p\int_\Sigma U^i\gamma^{p-1}\nabla_i\gamma d\mu \nonumber
\end{align}
where 
\begin{align}
\vec{T}&:= 2(\pi_{\vec{n}}\nabla_i\vec{H}\cdot \pi_{\vec{n}}\nabla_j\vec{H})\vec{h}^{ij} -4|\pi_{\vec{n}}\nabla\vec{H}|^2\vec{H} -4(\vec{H}\cdot\vec{h}_{ij})(\vec{H}\cdot\vec{h}^i_k)\vec{h}^{jk} +4 |\vec{H}\cdot\vec{h}|^2\vec{H} \nonumber
\\& \quad-\frac{1}{2}(\vec{H}\cdot\vec{h}^{ij})(\vec{h}_{kl}\cdot\vec{h}_{ij})\vec{h}^{kl}-28|\vec{H}|^4\vec{H} \nonumber
\end{align}
and 
\begin{align}
U^i&:=-\frac{1}{2}\Delta_\perp \vec{H}\cdot\nabla^i\vec{H} +\frac{1}{2} \vec{H}\cdot\nabla^i\Delta_\perp \vec{H} + \frac{1}{2}(\nabla^i\vec{H}\cdot\vec{h}_{jk})(\vec{H}\cdot\vec{h}^{jk})   -2(\nabla^j\vec{H}\cdot\vec{h}^i_j)|\vec{H}|^2 +\nabla^j|\vec{H}|^2(\vec{H}\cdot\vec{h}^i_j) \nonumber
\\& \quad + \frac{1}{2}\nabla^i((\vec{H}\cdot\vec{h}_{kj})\vec{h}^{kj}) \cdot\vec{H} +2 \nabla_j((\vec{H}\cdot\vec{h}^{ij})\vec{H})\cdot\vec{H}
-7\nabla^i(|\vec{H}|^2\vec{H})\cdot\vec{H}.  \nonumber
\end{align}
\end{prop}

\noindent
\begin{proof}
The proof follows by combining Lemmas \ref{firstlem}, \ref{seclem} and \ref{thirdlem}.
\end{proof}

\noindent
We will mostly rely on the Peter-Paul inequality to obtain our estimates in the remaining part of this paper.
\section{Curvature estimates and rigidity result}\label{three0009}
\noindent
We begin by estimating the terms on the right hand side of the integral identity in Proposition \ref{mekl}.
\begin{lem} \label{laks1}
Let $\vec{\Phi}:\Sigma\rightarrow \mathbb{R}^m$ be an immersion of a 4-dimensional manifold $\Sigma$ whose boundaries satisfy the condition \eqref{bdry} and let $\gamma$ be the cut-off function in \eqref{cutoff} then 
\begin{align}
\frac{1}{2}\int_{\Sigma} |\Delta_\perp\vec{H}|^2 \gamma^p d\mu             &\leq\int_\Sigma|\mathcal{\vec{W}}\cdot\vec{H}| \gamma^p d\mu  +\tilde{c} \int_\Sigma |\vec{h}_0|^2|\pi_{\vec{n}}\nabla\vec{H}|^2 \gamma^p d\mu  +\tilde {c}\int_\Sigma |\vec{H}|^2|\pi_{\vec{n}}\nabla\vec{H}|^2 \gamma^p d\mu \nonumber
\\&\quad +c\int_\Sigma |\vec{h}_0|^2|\pi_{\vec{n}}\nabla\vec{h}_0|^2 \gamma^p d\mu + \tilde{c}\int_\Sigma |\vec{H}|^4|\vec{h}_0|^2 \gamma^p d\mu + \tilde{c}\int_\Sigma |\vec{H}|^2|\vec{h}_0|^4 \gamma^p d\mu  \nonumber
\\& \quad + \tilde{c}\int_\Sigma |\vec{H}|^6 \gamma^p d\mu
+
p\int_\Sigma U^i\gamma^{p-1}\nabla_i\gamma d\mu.
\end{align}
\end{lem}
\begin{proof}
For the first estimate, we use the decomposition $\vec{h}=\vec{h}_0 +\vec{H}g$ to have
\begin{align}
\int_\Sigma \pi_{\vec{n}}\nabla^j\vec{H}\cdot\nabla_j[(\vec{H}\cdot\vec{h}^{ik})\vec{h}_{ik}]
\gamma^p d\mu &= \int_\Sigma |\pi_{\vec{n}}\nabla\vec{H}\cdot (\vec{h}_0)|^2
\gamma^p d\mu  \nonumber
+ \int_\Sigma (\pi_{\vec{n}}\nabla^j\vec{H}\cdot (\vec{h}_0)_{ik})(\vec{H}\cdot\pi_{\vec{n}}\nabla_j(\vec{h}_0)^{ik})
\gamma^p d\mu
  \\&\quad  + \int_\Sigma (\pi_{\vec{n}}\nabla^j\vec{H}\cdot \pi_{\vec{n}}\nabla_j(\vec{h}_0)_{ik})(\vec{H}\cdot(\vec{h}_0)^{ik})
\gamma^p d\mu  + \int_\Sigma \pi_{\vec{n}}\nabla\vec{H}*\nabla(|\vec{H}|^2\vec{H})\gamma^p d\mu \nonumber
\\ &\leq  \tilde{c} \int_\Sigma |\vec{h}_0|^2|\pi_{\vec{n}}\nabla\vec{H}|^2 \gamma^p d\mu   +\tilde {c}\int_\Sigma |\vec{H}|^2|\pi_{\vec{n}}\nabla\vec{H}|^2 \gamma^p d\mu +c\int_\Sigma |\vec{h}_0|^2|\pi_{\vec{n}}\nabla\vec{h}_0|^2 \gamma^p d\mu.\nonumber
\end{align}

\noindent
Next, we estimate
\begin{align}
\int_\Sigma (\nabla^j\vec{H}\cdot\vec{h}^k_j)(\vec{H}\cdot\nabla_k\vec{H})\gamma^p d\mu  \leq \frac{c}{2} \int_\Sigma |\vec{h}_0|^2|\pi_{\vec{n}}\nabla\vec{H}|^2 \gamma^p d\mu +\tilde{c} \int_\Sigma |\vec{H}|^2|\pi_{\vec{n}}\nabla\vec{H}|^2 \gamma^p d\mu. \nonumber
\end{align}

\noindent
Next, we obtain the estimate
\begin{align}
\int_\Sigma (\vec{H}\cdot\vec{h}^k_j)(\pi_{\vec{n}}\nabla^j\vec{H}\cdot\nabla_k\vec{H})\gamma^p d\mu           \leq \frac{c}{2} \int_\Sigma |\vec{h}_0|^2|\pi_{\vec{n}}\nabla\vec{H}|^2 \gamma^p d\mu +\tilde{c} \int_\Sigma |\vec{H}|^2|\pi_{\vec{n}}\nabla\vec{H}|^2 \gamma^p d\mu. \nonumber
\end{align}

\noindent
We split into normal and tangential components and then estimate
\begin{align}
\int_\Sigma\nabla_j((\vec{H}\cdot\vec{h}^{ij})\vec{H})\cdot\nabla_i\vec{H}
\gamma^p d\mu &= \int_\Sigma\pi_{\vec{n}}\nabla_j((\vec{H}\cdot\vec{h}^{ij})\vec{H})\cdot\pi_{\vec{n}}\nabla_i\vec{H}
\gamma^p d\mu  +\int_\Sigma\pi_T\nabla_j((\vec{H}\cdot\vec{h}^{ij})\vec{H})\cdot\pi_T\nabla_i\vec{H}
\gamma^p d\mu     \nonumber
\\& \leq \tilde {c} \int_\Sigma |\vec{h}_0|^2|\pi_{\vec{n}}\nabla\vec{H}|^2 \gamma^p d\mu +\tilde{c} \int_\Sigma |\vec{H}|^2|\pi_{\vec{n}}\nabla\vec{H}|^2 \gamma^p d\mu +\frac{c}{2}\int_\Sigma |\vec{H}|^2|\vec{h}_0|^4 \gamma^p d\mu \nonumber
\\& \quad \tilde {c}\int_\Sigma |\vec{H}|^4|\vec{h}_0|^2 \gamma^p d\mu \nonumber
\end{align}

\noindent
Next, we estimate
\begin{align}
\int_\Sigma \nabla^j(|\vec{H}|^2\vec{H})\cdot\pi_{\vec{n}}\nabla_j\vec{H} \gamma^p d\mu& \leq \tilde{c}\int_\Sigma |\vec{H}|^2 |\pi_{\vec{n}}\nabla\vec{H}|^2 \gamma^p d\mu  \nonumber                
\end{align}

\noindent
Finally, we estimate
\begin{align}
\int_\Sigma \vec{T}\cdot\vec{H} \gamma^p d\mu &\leq\tilde{c} \int_\Sigma |\vec{H}|^2  |\pi_{\vec{n}}\nabla\vec{H}|^2 \gamma^p d\mu   +\tilde{c} \int_\Sigma |\vec{h}_0|^2|\pi_{\vec{n}}\nabla\vec{H}|^2 \gamma^p d\mu +\tilde{c}\int_\Sigma |\vec{H}|^4|\vec{h}_0|^2 \gamma^p d\mu \nonumber
\\& \quad+\tilde{c} \int_\Sigma |\vec{H}|^2|\vec{h}_0|^4 \gamma^p d\mu + \tilde{c}\int_\Sigma |\vec{H}|^6 \gamma^p d\mu .  \nonumber
\end{align}

\noindent
Combining the six estimates above and using Proposition \ref{mekl} yields the desired result.

\end{proof}

\noindent
The following lemma is similar to Lemma 4 in \cite{glen}. However, we include here some extra terms resulting from correctly interchanging covariant derivatives.

\noindent
\begin{lem}\label{laks2}
Let $\vec{\Phi}:\Sigma\rightarrow \mathbb{R}^m$ be an immersion of a 4-dimensional manifold $\Sigma$ whose boundaries satisfy the condition \eqref{bdry} and let $\gamma$ be the cut-off function in \eqref{cutoff} then 
\begin{align}
\int_\Sigma (|\pi_{\vec{n}}\nabla_{(2)}\vec{H}|^2  + |\vec{H}|^2|\pi_{\vec{n}}\nabla\vec{H}|^2)\gamma^p d\mu 
&\leq c\int_\Sigma |\Delta_\perp\vec{H}|^2 \gamma^p d\mu+c \int_\Sigma |\vec{H}|^2|\vec{h}|^4 \gamma^p d\mu \nonumber
\\& + c c_\gamma^2 \int_\Sigma |\pi_{\vec{n}}\nabla \vec{h}_0|^2 \gamma^{p-2} d\mu + c\int_\Sigma |\vec{h}_0|^2 |\pi_{\vec{n}}\nabla\vec{H}|^2\gamma^p d\mu       .     \nonumber
\end{align}

\end{lem}

\begin{proof}
\noindent
We begin by interchanging covariant derivatives (see Appendix \ref{interchange} for full computation)
\begin{align}
\Delta_\perp\pi_{\vec{n}}\nabla_k\vec{H} &=g^{ij}\pi_{\vec{n}}\nabla_i\pi_{\vec{n}} \nabla_j\pi_{\vec{n}}\nabla_k\vec{H} \nonumber
\\&= \pi_{\vec{n}}\nabla_k\Delta\vec{H} +3|\vec{H}|^2\pi_{\vec{n}}\nabla_k\vec{H}+(\vec{h}* \vec{h}_0)*\pi_{\vec{n}}\nabla\vec{H}- \pi_{\vec{n}}\Delta\pi_{T} \nabla_k\vec{H} -\pi_{\vec{n}}\nabla^j\pi_T\nabla_j\pi_{\vec{n}} \nabla_k\vec{H}  \label{kingsl}
\end{align}
where we have use the *-notation to indicate contractions with the metric up to some constant.

\noindent
Using the identity 
\begin{align}
\Delta\vec{H}= \Delta_\perp\vec{H} -(\vec{H}\cdot\vec{h}_{ij})\vec{h}^{ij} -2\nabla_s\vec{H}\cdot\vec{h}^{sk}\nabla_k\vec{\Phi}  -2\nabla^k|\vec{H}|^2\nabla_k\vec{\Phi}  \nonumber
\end{align}
in \eqref{kingsl}, we obtain 
\begin{align}
\Delta_\perp\pi_{\vec{n}}\nabla_k\vec{H} &= \pi_{\vec{n}}\nabla_k\Delta_\perp\vec{H} +3|\vec{H}|^2\pi_{\vec{n}}\nabla_k\vec{H}+(\vec{h}* \vec{h}_0)*\pi_{\vec{n}}\nabla\vec{H} +\vec{U}_k
 \label{land}
\end{align}
where
\begin{align}
\vec{U}_k:=  -\pi_{\vec{n}}\nabla_k[(\vec{H}\cdot\vec{h}_{ij})\vec{h}^{ij}] -2(\nabla_s\vec{H}\cdot\vec{h}^{sj})\vec{h}_{jk} -2\nabla^i|\vec{H}|^2\vec{h}_{ik}- \pi_{\vec{n}}\Delta\pi_{T} \nabla_k\vec{H} -\pi_{\vec{n}}\nabla^j\pi_T\nabla_j\pi_{\vec{n}} \nabla_k\vec{H}.                                        \label{you}
\end{align}

\noindent
Integrating \eqref{land} against $\pi_{\vec{n}}\nabla^k\vec{H}\gamma^p$ gives
\begin{align}
\int_\Sigma\pi_{\vec{n}}\nabla_k\Delta_\perp\vec{H} \cdot \pi_{\vec{n}}\nabla^k\vec{H}\gamma^p d\mu &=\int_\Sigma\Delta_\perp\pi_{\vec{n}}\nabla_k\vec{H}\cdot \pi_{\vec{n}}\nabla^k\vec{H}\gamma^p d\mu
 -3\int_\Sigma|\vec{H}|^2|\pi_{\vec{n}}\nabla\vec{H}|^2 d\mu \nonumber
\\&\quad+\int_\Sigma(\vec{h}* \vec{h}_0)*\pi_{\vec{n}}\nabla\vec{H} \cdot \pi_{\vec{n}}\nabla\vec{H}\gamma^pd\mu-\int_\Sigma\vec{U}_k\cdot \pi_{\vec{n}}\nabla^k\vec{H}\gamma^p d\mu \label{land11}
\end{align}

\noindent
Divergence theorem and boundary condition \eqref{bdry} give
\begin{align}
\int_\Sigma\pi_{\vec{n}}\nabla_k\Delta_\perp\vec{H} \cdot \pi_{\vec{n}}\nabla^k\vec{H}\gamma^p d\mu &+\int_\Sigma |\Delta_\perp\vec{H}|^2\gamma^p d\mu 
+p\int_\Sigma \Delta_\perp\vec{H}\cdot\pi_{\vec{n}}\nabla^k\vec{H} \gamma^{p-1}\nabla_k\gamma  d\mu \nonumber
\\&=\int_{\partial\Sigma} \Delta_\perp\vec{H}\cdot \pi_{\vec{\eta}}\nabla\vec{H}\gamma^p d\sigma =0 \label{land22}
\end{align}

\noindent
On the other hand, divergence theorem and boundary condition \eqref{bdry} also give
\begin{align}
\int_\Sigma \Delta_\perp\pi_{\vec{n}}\nabla_k\vec{H}\cdot \pi_{\vec{n}}\nabla^k\vec{H} \gamma^p d\mu &+\int_\Sigma |\pi_{\vec{n}}\nabla_{(2)}\vec{H}|^2 \gamma^p d\mu + p\int_{\Sigma} \pi_{\vec{n}}\nabla^i\pi_{\vec{n}}\nabla_k\vec{H}\cdot\pi_{\vec{n}}\nabla^k\vec{H}
\gamma^{p-1}\nabla_i \gamma d\mu  \nonumber
\\&=\frac{1}{2}\int_{\partial\Sigma} \pi_{\vec{\eta}} \nabla|\pi_{\vec{n}}\nabla\vec{H}|^2 \gamma^p d\sigma =0  \label{land33}
\end{align}
Accordingly, by combining \eqref{land11}, \eqref{land22}, \eqref{land33} and calling upon the Cauchy-Schwartz and Peter-Paul inequalities we find for $c >0$
\begin{align}
\int_\Sigma |\pi_{\vec{n}}\nabla_{(2)}\vec{H}|^2 \gamma^p d\mu +3\int_\Sigma |\vec{H}|^2|\pi_{\vec{n}}\nabla\vec{H}|^2\gamma^p d\mu &= \int_\Sigma |\Delta_\perp\vec{H}|^2 \gamma^p d\mu +\int_\Sigma \pi_{\vec{n}}\nabla_{(2)}\vec{H}*\pi_{\vec{n}}\nabla\vec{H}*\nabla\gamma \gamma^{p-1} d\mu\nonumber
\\&+    \int_\Sigma \vec{h}* (\vec{h}_0)*\pi_{\vec{n}}\nabla\vec{H}*\pi_{\vec{n}}\nabla\vec{H} \gamma^p d\mu  +\int_{\Sigma} \vec{U}* \pi_{\vec{n}}\nabla\vec{H}\gamma^p d\mu               \nonumber
\\&\leq \int_\Sigma |\Delta_\perp\vec{H}|^2 \gamma^p d\mu+\frac{1}{2c} \int_\Sigma |\pi_{\vec{n}}\nabla_{(2)}\vec{H}|^2 \gamma^p d\mu \nonumber
\\& + \frac{c}{2} c_\gamma^2 \int_\Sigma |\pi_{\vec{n}}\nabla \vec{h}_0|^2 \gamma^{p-2} d\mu + \frac{1+2c^2}{2c}\int_\Sigma |\vec{h}_0|^2 |\pi_{\vec{n}}\nabla\vec{H}|^2\gamma^p d\mu  \nonumber
\\& +\frac{1}{2c} \int_\Sigma |\vec{H}|^2|\pi_{\vec{n}}\nabla\vec{H}|^2 \gamma^pd\mu +\int_{\Sigma} \vec{U}* \pi_{\vec{n}}\nabla\vec{H}\gamma^p d\mu       .     \label{lab1}
\end{align}

\noindent
Next, by the same token, \eqref{you} gives
\begin{align}
\int_\Sigma\vec{U}_k\cdot\pi_{\vec{n}}\nabla^k\vec{H} \gamma^p d\mu  &=\int_\Sigma (\vec{H}*\vec{h})*(\vec{h}*\pi_{\vec{n}}\nabla_{(2)}\vec{H})  \gamma^p d\mu  + \int_\Sigma (\vec{H}*\vec{h})*(\vec{h}*\pi_{\vec{n}}\nabla\vec{H})  \nabla\gamma\gamma^{p-1} d\mu 
 \nonumber
\\&+\int_\Sigma (\pi_{\vec{n}}\nabla\vec{H}*(\vec{h}_0+\vec{H}g))*((\vec{h}_0+\vec{H}g)*\pi_{\vec{n}}\nabla\vec{H})  \gamma^p d\mu   \nonumber
\\&\leq \frac{c}{2}\int_\Sigma |\vec{H}|^2|\vec{h}|^4 \gamma^p d\mu +\frac{1}{2c}\int_\Sigma |\pi_{\vec{n}}\nabla_{(2)} \vec{H}|^2 \gamma^p d\mu +\frac{1}{2c}\int_\Sigma |\vec{H}|^2|\vec{h}|^4 \gamma^p d\mu \nonumber
\\&+ \frac{c}{2}c^2_\gamma \int_\Sigma |\pi_{\vec{n}}\nabla\vec{h}_0|^2 \gamma^{p-2} d\mu +\frac{1+2c^2}{2c} \int_\Sigma |\vec{h}_0|^2|\pi_{\vec{n}}\nabla\vec{H}|^2\gamma^p d\mu +\frac{2+c^2}{2c}\int_\Sigma |\vec{H}|^2|\pi_{\vec{n}}\nabla\vec{H}|^2\gamma^p d\mu.  \label{lastt}
\end{align}

\noindent
The proof is finished by substituting \eqref{lastt} into \eqref{lab1} and absorbing.

\end{proof}

\noindent
If we combine the estimates in Lemmas \ref{laks1} and \ref{laks2}, we obtain the following result.

\noindent
\begin{lem}\label{okl1}
Let $\vec{\Phi}:\Sigma\rightarrow \mathbb{R}^m$ be an immersion of a 4-dimensional manifold $\Sigma$ whose boundaries satisfy the condition \eqref{bdry} and let $\gamma$ be the cut-off function in \eqref{cutoff} then 
\begin{align}
\int_\Sigma |\pi_{\vec{n}}\nabla_{(2)}\vec{H}|^2 \gamma^p d\mu &\leq c \int_\Sigma |\mathcal{\vec{W}}\cdot\vec{H}| \gamma^p d\mu   + c\int_\Sigma |\vec{h}_0|^2  |\pi_{\vec{n}}\nabla\vec{H}|^2 \gamma^p d\mu + c\int_\Sigma |\vec{H}|^2  |\pi_{\vec{n}}\nabla\vec{H}|^2 \gamma^p d\mu     \nonumber
\\&\quad   + c\int_\Sigma |\vec{h}_0|^2  |\pi_{\vec{n}}\nabla\vec{h}_0|^2 \gamma^p d\mu                    + c\int_\Sigma |\vec{H}|^2  |\vec{h}|^4 \gamma^p d\mu \nonumber + cc_\gamma^2\int_\Sigma  |\pi_{\vec{n}}\nabla\vec{h}_0|^2 \gamma^{p-2} d\mu     \nonumber
\\&\quad\quad + c\int_\Sigma |\vec{H}|^4  |\vec{h}|^2 \gamma^p d\mu   + c\int_\Sigma |\vec{H}|^6   \gamma^p d\mu \nonumber   + 2p\int_\Sigma U^i \gamma^{p-1} \nabla_i\gamma d\mu           . \nonumber
\end{align}

\end{lem}
\noindent
We now establish estimates involving the derivatives of the trace-free second fundamental form. To do this we need a Simon-type indentity for $\vec{h}_0$.

\noindent
\begin{cl}[A Simon-type identity]
We have 
\begin{align}
\Delta_\perp \vec{h}_0  &= (4\pi_{\vec{n}} \nabla \pi_{\vec{n}} \nabla\vec{H} -g\Delta_\perp\vec{H} )+ (\vec{h}_0 * \vec{h}_0)*\vec{h}_0 +(\vec{H}*\vec{h}_0)*\vec{h}_0 + 4|\vec{H}|^2\vec{h}_0   . \label{simonid}
\end{align}
\end{cl}
\begin{proof}
Observe that the Codazzi-Mainardi equations and interchanging covariant derivatives imply
\begin{align}
\Delta_\perp \vec{h}_{jk} &= \nabla_i\nabla_j \vec{h}^i_k -\nabla_i \pi_T \nabla_j \vec{h}^i_k -\pi_T\nabla_i\pi_{\vec{n}} \nabla_j\vec{h}^i_k   \nonumber
\\&= 4\pi_{\vec{n}} \nabla_j \pi_{\vec{n}} \nabla_k\vec{H} + 4\pi_T \nabla_j\pi_{\vec{n}} \nabla_k\vec{H} + \nabla_j \pi_T \nabla_i \vec{h}^i_k   \nonumber
\\&\quad+\tensor {R}{_{ij}^i_e}\vec{h}^e_k + \tensor {R}{_{ijke}}\vec{h}^{ie} -\nabla_i \pi_T \nabla_j \vec{h}^i_k -\pi_T\nabla_i\pi_{\vec{n}} \nabla_j\vec{h}^i_k   \nonumber
\\&= 4\pi_{\vec{n}} \nabla_j \pi_{\vec{n}} \nabla_k\vec{H}  +4(\vec{H}\cdot\vec{h}_{ij})\vec{h}^i_k  -(\vec{h}_{ie}\cdot\vec{h}_{jk})\vec{h}^{ie}
  \nonumber
\\&\quad+2 (\vec{h}^i_k\cdot\vec{h}_{je})\vec{h}_i^e - (\vec{h}^i_k\cdot\vec{h}_{ie})\vec{h}^e_j -(\vec{h}_{ie}\cdot\vec{h}^i_j )\vec{h}^e_k .\label{symm6}
\end{align}
Accordingly, we use the decomposition $\vec{h}_0= \vec{h}-g\vec{H}$ to arrive at 
\begin{align}
\Delta_\perp \vec{h}_0  &= (4\pi_{\vec{n}} \nabla \pi_{\vec{n}} \nabla\vec{H} -g\Delta_\perp\vec{H} )+ (\vec{h}_0 * \vec{h}_0)*\vec{h}_0 +(\vec{H}*\vec{h}_0)*\vec{h}_0 + 4|\vec{H}|^2\vec{h}_0   . \nonumber
\end{align}

\end{proof}

\noindent
\begin{lem} \label{okl}
Let $\vec{\Phi}:\Sigma\rightarrow \mathbb{R}^m$ be an immersion of a 4-dimensional manifold $\Sigma$ whose boundaries satisfy the condition \eqref{bdry} and let $\gamma$ be the cut-off function in \eqref{cutoff} then 
\begin{align}
&\int_\Sigma |{\pi_{\vec{n}}} \nabla \vec{h}_0|^2 |\vec{H}|^2 \gamma^p d\mu + \int_\Sigma |\vec{H}|^4 |\vec{h}_0| ^2 \gamma^p d\mu \nonumber 
\\& \leq c \int_\Sigma |\vec{H}|^2 |\pi_{\vec{n}}\nabla\vec{H}|^2 \gamma^p d\mu   +c\int_\Sigma |\vec{h}_0|^2 |\pi_{\vec{n}}\nabla\vec{H}|^2 \gamma^p d\mu +c \int_\Sigma  |\vec{h}_0|^6 \gamma^p d\mu    \nonumber
\\& \quad 
 + c\int_\Sigma |\vec{H}|^2 |\vec{h}_0|^4 \gamma^p d\mu       
+c  \int_\Sigma |\vec{H}|^4 \gamma^{p} d\mu    
  +cc^4_{\gamma} \int_\Sigma |\vec{h}_0|^4 \gamma^{p-4} d\mu    .     
\end{align}

\end{lem}

\noindent
\begin{proof}
First, we observe that the boundary condition \eqref{bdry} implies
\begin{align}
&\int_\Sigma |\vec{H}|^2 \Delta_\perp(\vec{h}_0)_{jk} \cdot (\vec{h}_0)^{jk} \gamma^p d\mu +  \int_\Sigma |{\pi_{\vec{n}}} \nabla \vec{h}_0|^2 |\vec{H}|^2 \gamma^p d\mu  \nonumber   + 2 \int_\Sigma (\vec{h}_0)^{jk}\cdot\pi_{\vec{n}} \nabla^i(\vec{h}_0)_{jk} \vec{H}\cdot\nabla_i\vec{H} \gamma^p d\mu  \nonumber
\\& \quad
 +p\int_\Sigma (\vec{h}_0)^{jk} \cdot \pi_{\vec{n}}\nabla^i (\vec{h}_0)_{jk} \nabla_i \gamma \gamma^{p-1} d\mu  \nonumber
\\& =  \int_{\partial \Sigma} (\vec{h}_0)^{jk} \cdot \pi_{\vec{\eta}} \nabla_i (\vec{h}_0)_{jk} |\vec{H}|^2 \gamma^p d\sigma =0.  \label{useme}
\end{align}

\noindent
Using the Simon-type identity  \eqref{simonid} in \eqref{useme} gives 
\begin{align}
&\int_\Sigma |{\pi_{\vec{n}}} \nabla \vec{h}_0|^2 |\vec{H}|^2 \gamma^p d\mu +4 \int_\Sigma |\vec{H}|^4 |\vec{h}_0| ^2 \gamma^p d\mu \nonumber  
\\&=-4\int_\Sigma |\vec{H}|^2  (\vec{h}_0)^{jk}\cdot\pi_{\vec{n}}\nabla_j\pi_{\vec{n}} \nabla_k\vec{H} \gamma^p d\mu  \nonumber +\int_\Sigma |\vec{H}|^2 (\vec{h}_0 *\vec{h}_0) * (\vec{h}_0 *\vec{h}_0) \gamma^p d\mu
+\int_\Sigma|\vec{H}|^2 (\vec{H} *\vec{h}_0*\vec{h}_0*\vec{h}_0) \gamma^p d\mu
\\& \quad    -2 \int_\Sigma (\vec{h}_0)^{jk}\cdot\pi_{\vec{n}} \nabla^i(\vec{h}_0)_{jk} \vec{H}\cdot\nabla_i\vec{H} \gamma^p d\mu   -p\int_\Sigma (\vec{h}_0)^{jk} \cdot \pi_{\vec{n}}\nabla^i (\vec{h}_0)_{jk} \nabla_i \gamma \gamma^{p-1} d\mu       \nonumber
\\&=     8\int_\Sigma (\vec{H}\cdot\nabla_j\vec{H}) (\vec{h}_0)^{jk}\cdot\pi_{\vec{n}}\nabla_k\vec{H}  \gamma^p d\mu +12 \int_{\Sigma} |\vec{H}|^2 |\pi_{\vec{n}}\nabla \vec{H}|^2 \gamma^p d\mu  + \int_\Sigma |\vec{H}|^2 (\vec{h}_0 *\vec{h}_0) * (\vec{h}_0 *\vec{h}_0)   \gamma^p d\mu               \nonumber
\\&\quad + \int_\Sigma |\vec{H}|^2(\vec{H} *\vec{h}_0*\vec{h}_0*\vec{h}_0)    \gamma^p d\mu    -2 \int_\Sigma (\vec{h}_0)^{jk}\cdot\pi_{\vec{n}}\nabla^i (\vec{h}_0)_{jk} \vec{H}\cdot\nabla_i\vec{H} \gamma^p d\mu           \nonumber   
\\
& \quad+ p\int_\Sigma|\vec{H}|^2 \left[4 (\vec{h}_0)^{ij}\cdot\nabla_j\vec{H}  -(\vec{h}_0)^{jk}\cdot \pi_{\vec{n}} \nabla^i (\vec{h}_0)_{jk}   \right]   \nabla_i\gamma \gamma^{p-1}        d\mu     \nonumber
\end{align}

\noindent
where we have used the boundary condition \eqref{bdry} to obtain the second equality above. Next, for $\delta>0$ we estimate
\begin{align}
&\int_\Sigma |{\pi_{\vec{n}}} \nabla \vec{h}_0|^2 |\vec{H}|^2 \gamma^p d\mu +4 \int_\Sigma |\vec{H}|^4 |\vec{h}_0| ^2 \gamma^p d\mu \nonumber  
\\& \leq (\delta+12) \int_\Sigma |\vec{H}|^2 |\pi_{\vec{n}}\nabla\vec{H}|^2 \gamma^p d\mu   +\frac{c}{\delta} \int_\Sigma |\vec{h}_0|^2 |\pi_{\vec{n}}\nabla\vec{H}|^2 \gamma^p d\mu +c\delta \int_\Sigma |\vec{H}|^4 |\vec{h}_0|^2 \gamma^p d\mu  \nonumber
\\& \quad +\frac{c}{\delta} \int_\Sigma  |\vec{h}_0|^6 \gamma^p d\mu  
 + \frac{c}{\delta} \int_\Sigma |\vec{H}|^2 |\vec{h}_0|^4 \gamma^p d\mu       + \delta \int_\Sigma |\vec{H}|^2 |\pi_{\vec{n}}\nabla\vec{h}_0|^2 \gamma^p d\mu               
  +\frac{cc^2_{\gamma}}{\delta} \int_\Sigma |\vec{H}|^2 |\vec{h}_0|^2 \gamma^{p-2} d\mu    .               \label{lakk1}
\end{align}

\noindent
We next estimate
\begin{align}
\frac{cc^2_{\gamma}}{\delta} \int_\Sigma |\vec{H}|^2 |\vec{h}_0|^2 \gamma^{p-2} d\mu \leq    \delta\int_\Sigma |\vec{H}|^4  \gamma^p d\mu  + \frac{cc_\gamma^4}{\delta^3} \int_\Sigma |\vec{h}_0|^4 \gamma^{p-4} d\mu                . \label{lakk}
\end{align}
The proof is finished by using \eqref{lakk} in \eqref{lakk1} and absorbing. 

\end{proof}

\noindent
\begin{lem}  \label{u22}
Let $\vec{\Phi}:\Sigma\rightarrow \mathbb{R}^m$ be an immersion of a 4-dimensional manifold $\Sigma$ whose boundaries satisfy the condition \eqref{bdry} and let $\gamma$ be the cut-off function in \eqref{cutoff} then 
\begin{align}
&\int_\Sigma |\pi_{\vec{n}}\nabla_{(2)}\vec{H}|^2  \gamma^p d\mu+\int_\Sigma |\vec{h}_0|^2 |\vec{h}|^4 \gamma^p d\mu   \nonumber
\\&\leq  c\int_\Sigma |\vec{h}_0|^2 |\pi_{\vec{n}}\nabla\vec{h}_0|^2 \gamma^p d\mu   + c\int_\Sigma |\vec{h}_0|^6 \gamma^p d\mu
+ c \int_\Sigma |\vec{H}|^4  \gamma^p d\mu  + cc_\gamma^4 \int_\Sigma |\vec{h}_0|^4  \gamma^{p-4} d\mu       \nonumber
\\& \quad + c \int_\Sigma |\mathcal{\vec{W}}\cdot \vec{H}| \gamma^p d\mu   + cc^2_\gamma \int_\Sigma |\pi_{\vec{n}}\nabla\vec{h}_0|^2 \gamma^{p-2} d\mu            + 2p\int_\Sigma U^i \gamma^{p-1} \nabla_i \gamma  \,d\mu.               \nonumber
\end{align}
\end{lem}

\noindent
\begin{proof}
We first observe that the decomposition $\vec{h}_0= \vec{h}-\vec{H}g$ yields the estimate
\begin{align}
\int_\Sigma |\vec{h}_0|^2 |\vec{h}|^4 \gamma^p d\mu
\leq c \int_\Sigma |\vec{h}_0|^6 \gamma^p d\mu+ c\int_\Sigma  |\vec{H}|^4 |\vec{h}_0|^2 \gamma^p d\mu .    \nonumber 
\end{align}
Combining the latter with Lemma \ref{okl1} and Lemma \ref{okl}, and using  the estimates
\begin{align}
|\vec{H}|\leq |\vec{h}|\leq c|\vec{h}_0|   \quad\mbox{and}\quad |\nabla\vec{H}|\leq |\nabla\vec{h}|\leq c|\nabla\vec{h}_0| ,                 \label{est12}
\end{align}

\noindent
we obtain
\begin{align}
&\int_\Sigma |\pi_{\vec{n}}\nabla_{(2)}\vec{H}|^2  \gamma^p d\mu+\int_\Sigma |\vec{h}_0|^2 |\vec{h}|^4 \gamma^p d\mu   \nonumber
\\&\leq  c \int_\Sigma |\mathcal{\vec{W}}\cdot \vec{H}| \gamma^p d\mu + c\int_\Sigma |\vec{h}_0|^2 |\pi_{\vec{n}}\nabla\vec{h}_0|^2 \gamma^p d\mu   + c\int_\Sigma |\vec{h}_0|^6 \gamma^p d\mu
+ c \int_\Sigma |\vec{H}|^4  \gamma^p d\mu         \nonumber
\\& \quad + cc_\gamma^4 \int_\Sigma |\vec{h}_0|^4  \gamma^{p-4} d\mu   + cc^2_\gamma \int_\Sigma |\pi_{\vec{n}}\nabla\vec{h}_0|^2 \gamma^{p-2} d\mu            + 2p\int_\Sigma U^i \gamma^{p-1} \nabla_i \gamma  \,d\mu.               \nonumber
\end{align}

\end{proof}

\noindent
\begin{lem} \label{noncy1}
Let $\vec{\Phi}:\Sigma\rightarrow \mathbb{R}^m$ be an immersion of a 4-dimensional manifold $\Sigma$ whose boundaries satisfy the condition \eqref{bdry} and let $\gamma$ be the cut-off function in \eqref{cutoff} then 
\begin{align}
&\int_\Sigma (|\pi_{\vec{n}}\nabla_{(2)}\vec{h}|^2 +|\vec{h}|^4|\vec{h}_0|^2 +|\vec{h}|^2|\pi_{\vec{n}}\nabla\vec{h}|^2 ) \gamma^p d\mu                                    \nonumber
\\& \quad \leq  c\int_\Sigma |\mathcal{\vec{W}}\cdot \vec{H}| \gamma^p d\mu + c\int_{\Sigma} |\vec{h}_0|^2|\pi_{\vec{n}}\nabla\vec{h}_0|^2 \gamma^p d\mu  + c\int_\Sigma |\vec{h}|^6 \gamma^p d\mu  +c\int_\Sigma |\vec{H}|^4 \gamma^p d\mu                   \nonumber
\\&\quad\quad + cc_\gamma^4 \int_\Sigma |\vec{h}|^4 \gamma^{p-4} d\mu + cc_\gamma^2 \int_\Sigma |\pi_{\vec{n}}\nabla\vec{h}_0|^2 \gamma^{p-2} d\mu       + 2p\int_\Sigma U^i\gamma^{p-1} \nabla_i \gamma d\mu       .                      \label{noncy}
\end{align}
\end{lem}

\noindent
\begin{proof}
We begin by interchanging covariant derivatives (see Appendix \ref{interchange} for full computation)
\begin{align}
\Delta_\perp\pi_{\vec{n}} \nabla_j \vec{h}_{kl} &    =   \pi_{\vec{n}} \nabla_j \Delta_\perp \vec{h}_{kl} + \pi_{T} \nabla_j \Delta_\perp \vec{h}_{kl} +\nabla_j \pi_T \nabla_i \pi_{\vec{n}} \nabla^i \vec{h}_{kl}      +  \nabla_j \nabla_i \pi_{T} \nabla^i \vec{h}_{kl}  \nonumber
\\&   \quad\quad+ (\vec{h}* \pi_{\vec{n}}\nabla\vec{h} * \vec{h} + \vec{h}*  \vec{h}*\pi_{\vec{n}}\nabla\vec{h} )_{jkl}   \nonumber
\\& \quad \quad - \pi_{T} \Delta  \nabla_j \vec{h}_{kl}    - \pi_{\vec{n}} \nabla_i\pi_{T} \nabla^i \nabla_j \vec{h}_{kl}  - \pi_{\vec{n}} \nabla_i\pi_{\vec{n}} \nabla^i \pi_{T} \nabla_j \vec{h}_{kl}  . \label{multi}
\end{align}

\noindent
Next, we integrate by parts each term of \eqref{multi} against $\pi_{\vec{n}}\nabla\vec{h}  \gamma^p$,  use the divergence theorem together with the boundary conditions \eqref{bdry} and then apply the Peter-Paul inequality. We have for $\delta >0$  the following estimates
\begin{align}
 \int_\Sigma |\pi_{\vec{n}}\nabla_{(2)}\vec{h}|^2 \gamma^p d\mu               &=-\int_\Sigma \Delta_\perp \pi_{\vec{n}} \nabla_j \vec{h}_{kl} \cdot \pi_{\vec{n}} \nabla^j \vec{h}^{kl} \gamma^p d\mu   +\int_\Sigma \pi_{\vec{n}} \nabla_{(2)} \vec{h}  *\pi_{\vec{n}} \nabla\vec{h} \nabla\gamma^p d\mu                     \nonumber
\\&\leq -\int_\Sigma \Delta_\perp \pi_{\vec{n}} \nabla_j \vec{h}_{kl} \cdot \pi_{\vec{n}} \nabla^j \vec{h}^{kl} \gamma^p d\mu  + \delta \int_\Sigma |\pi_{\vec{n}}\nabla_{(2)}\vec{h}|^2 \gamma^p d\mu + \frac{cc_\gamma^2}{\delta}\int_\Sigma |\pi_{\vec{n}}\nabla\vec{h}|^2 \gamma^{p-2} d\mu               \nonumber
\end{align}
and
\begin{align}
\int_\Sigma \pi_{\vec{n}} \nabla_j \Delta_\perp \vec{h}_{kl} \cdot \pi_{\vec{n}} \nabla^j \vec{h}^{kl} \gamma^p d\mu 
&= -\int_\Sigma |\Delta_\perp \vec{h}|^2 \gamma^p d\mu - \int_\Sigma \Delta_\perp \vec{h}_{kl} \cdot \pi_{\vec{n}} \nabla^j \vec{h}^{kl} \nabla_j\gamma^p d\mu     \nonumber
\\& \leq    c \int_\Sigma |\Delta_\perp \vec{h}|^2 \gamma^p d\mu  +cc_\gamma^2 \int_\Sigma |\pi_{\vec{n}}\nabla\vec{h}|^2 \gamma^{p-2} d\mu      .     \nonumber
\end{align}
In a similar fashion, we have the estimate
\begin{align}
\int_\Sigma \nabla_j \pi_T \nabla_i \pi_{\vec{n}} \nabla^i \vec{h}_{kl} \cdot \pi_{\vec{n}} \nabla^j \vec{h}^{kl} \gamma^p d\mu & = -\int_\Sigma \pi_T \nabla_i \pi_{\vec{n}} \nabla^i \vec{h}_{kl}\cdot   \pi_T \nabla_j \pi_{\vec{n}} \nabla^j \vec{h}_{kl}           \gamma^p d\mu                   \nonumber
\\&= -\int_\Sigma (\vec{h}_{im}\cdot \pi_{\vec{n}}\nabla^i \vec{h}_{kl}) (\vec{h}_j^m \cdot\pi_{\vec{n}}\nabla^j \vec{h}^{kl}) \gamma^p d\mu                         \nonumber
\\& \leq c\int_\Sigma |\vec{h}|^2 |\pi_{\vec{n}}\nabla\vec{h}|^2    \gamma^p d\mu               \nonumber
\end{align}
and
\begin{align}
\int_\Sigma  \nabla_j  \nabla_i \pi_T \nabla^i \vec{h}_{kl} \cdot \pi_{\vec{n}} \nabla^j\vec{h}^{kl} \gamma^p d\mu 
&= -\int_\Sigma \nabla_i\pi_T \nabla^i \vec{h}_{kl} \cdot \nabla_j \pi_{\vec{n}} \nabla^j \vec{h}^{kl} \gamma^p d\mu 
-p\int _\Sigma \nabla_i\pi_T \nabla^i \vec{h}_{kl} \cdot \pi_{\vec{n}} \nabla^j \vec{h}^{kl} \nabla_j\gamma \gamma^{p-1} d\mu  \nonumber  
\\& =  -\int_\Sigma (\vec{h}_{kl}\cdot \vec{h}^{im}) (\Delta_\perp \vec{h}^{kl}\cdot\vec{h}_{im}) \gamma^p d\mu + \int_\Sigma \nabla_i(\vec{h}_{kl}\cdot \vec{h}^{im}) (\pi_{\vec{n}}\nabla^j\vec{h}^{kl}\cdot\vec{h}_{jm})  \gamma^p d\mu                                \nonumber
\\&\quad \quad+  p\int_\Sigma (\vec{h}_{kl}\cdot\vec{h}^{im}) (\vec{h}_{im}\cdot\pi_{\vec{n}}\nabla^j\vec{h}^{kl}) \nabla_j \gamma \gamma^{p-1} d\mu                  \nonumber
\\&\leq      c\int_\Sigma |\vec{h}|^6 \gamma^p d\mu + c\int_\Sigma |\Delta_\perp\vec{h}|^2 \gamma^p d\mu +       c\int_\Sigma |\vec{h}|^2 |\pi_{\vec{n}}\nabla\vec{h}|^2    \gamma^p d\mu    \nonumber
\\&\quad\quad  +cc_\gamma^2 \int_\Sigma |\pi_{\vec{n}}\nabla\vec{h}| ^2 \gamma^{p-2} d\mu .               \nonumber
\end{align}
By the same token, we have the estimate
\begin{align}
\int_\Sigma \pi_{\vec{n}}\nabla_i \pi_T \nabla^i \nabla_j \vec{h}_{kl}\cdot \pi_{\vec{n}} \nabla^j\vec{h}^{kl} \gamma^p d\mu &= -\int_\Sigma \pi_T \nabla^i\nabla_j \vec{h}_{kl} \cdot \nabla_i \pi_{\vec{n}} \nabla^j\vec{h}^{kl}  \gamma^p d\mu                        \nonumber
\\& = -\int_\Sigma (\vec{h}^{im}\cdot\pi_{\vec{n}}\nabla^j\vec{h}^{kl}) (\vec{h}_{im}\cdot\pi_{\vec{n}}\nabla_j\vec{h}_{kl}) \gamma^p d\mu             \nonumber
\\& \leq          c\int_\Sigma |\vec{h}|^2 |\pi_{\vec{n}}\nabla\vec{h}|^2    \gamma^p d\mu              \nonumber
\end{align}
and
\begin{align}
\int_\Sigma \pi_{\vec{n}}\nabla_i \pi_{\vec{n}}\nabla^i \pi_T \nabla_j \vec{h}_{kl} \cdot\pi_{\vec{n}} \nabla^j \vec{h}^{kl} \gamma^p d\mu &= - \int_\Sigma \pi_{\vec{n}} \nabla^i \pi_T \nabla_j \vec{h}_{kl} \cdot \pi_{\vec{n}}\nabla_i\pi_{\vec{n}} \nabla^j \vec{h}^{kl} \gamma^p d\mu  \nonumber
\\&\quad\quad  - p\int_\Sigma \pi_{\vec{n}} \nabla^i \pi_T \nabla_j \vec{h}_{kl}\cdot \pi_{\vec{n}} \nabla^j \vec{h}^{kl} \nabla_i \gamma \gamma^{p-1} d\mu   \nonumber
\\& =  \int_\Sigma  (\vec{h}_{kl}\cdot\vec{h}_{jm}) (\vec{h}^{im}\cdot\pi_{\vec{n}}\nabla_i\pi_{\vec{n}}\nabla^j\vec{h}^{kl})  \gamma^p d\mu  \nonumber
\\&\quad\quad +p\int_\Sigma (\vec{h}_{kl}\cdot\vec{h}_{jm})(\vec{h}^{im}\cdot\pi_{\vec{n}}\nabla^j\vec{h}^{kl})\nabla_i \gamma \gamma^{p-1} d\mu   \nonumber 
\\&\leq  \frac{c}{\delta}\int_\Sigma |\vec{h}|^6 \gamma^p d\mu  + \delta\int_\Sigma |\pi_{\vec{n}}\nabla_{(2)}\vec{h}|^2 \gamma^p d\mu            \nonumber
+ cc_\gamma^2\int_\Sigma  |\pi_{\vec{n}}\nabla\vec{h}|^2 \gamma^{p-2} d\mu   .  \nonumber
\end{align}
Finally, we have the estimate
\begin{align}
\int_\Sigma (\vec{h}* \pi_{\vec{n}}\nabla\vec{h} * \vec{h} + \vec{h}*  \vec{h}*\pi_{\vec{n}}\nabla\vec{h} )_{jkl} \cdot\pi_{\vec{n}} \nabla^j \vec{h}^{kl} \gamma^p d\mu  \leq c\int_\Sigma |\vec{h}|^2 |\pi_{\vec{n}}\nabla\vec{h}|^2 \gamma^p d\mu     \nonumber
\end{align}

\noindent
Combining all the estimates above  and absorbing yields
\begin{align}
  \int_\Sigma |\pi_{\vec{n}}\nabla_{(2)}\vec{h}|^2 \gamma^p d\mu   
 &\leq c\int_\Sigma |\Delta_\perp \vec{h}   |^2 \gamma^p d\mu          \nonumber
      + c\int_\Sigma |\vec{h}|^2 |\pi_{\vec{n}}\nabla\vec{h}|^2    \gamma^p d\mu               
  +c\int_\Sigma |\vec{h}|^6 \gamma^p d\mu  +   cc_\gamma^2\int_\Sigma |\pi_{\vec{n}}\nabla\vec{h}|^2 \gamma^{p-2} d\mu    .           \nonumber                     
\end{align}

\noindent
The Simon-type identity \eqref{symm6}  implies
\begin{align}
|\Delta_\perp \vec{h}|^2 \leq c|\pi_{\vec{n}}\nabla_{(2)}\vec{H}|^2 + c|\vec{H}|^2 |\vec{h}|^4 + c|\vec{h}|^6  \nonumber
\end{align}
which allows us obtain the estimate
\begin{align}
&\int_\Sigma |\pi_{\vec{n}}\nabla_{(2)}\vec{h}|^2 \gamma^p d\mu   +\int_\Sigma|\vec{h}|^4 |\vec{h}_0|^2 \gamma^p d\mu +\int_\Sigma |\vec{h}|^2|\pi_{\vec{n}}\nabla\vec{h}|^2 \gamma^p d\mu \nonumber
 \\&\leq c\int_\Sigma |\pi_{\vec{n}}\nabla_{(2)}\vec{H}|^2 \gamma^p d\mu          \nonumber
      + c\int_\Sigma |\vec{h}|^2 |\pi_{\vec{n}}\nabla\vec{h}|^2    \gamma^p d\mu + c\int_\Sigma |\vec{H}|^2 |\vec{h}|^4 \gamma^p d\mu    \nonumber
\\&  \quad          
  +c\int_\Sigma |\vec{h}|^6 \gamma^p d\mu  +   cc_\gamma^2\int_\Sigma |\pi_{\vec{n}}\nabla\vec{h}|^2 \gamma^{p-2} d\mu              
 +\int_\Sigma|\vec{h}|^4 |\vec{h}_0|^2 \gamma^p d\mu +\int_\Sigma |\vec{h}|^2|\pi_{\vec{n}}\nabla\vec{h}|^2 \gamma^p d\mu  .  \nonumber
\end{align}
\noindent
Using  Lemma \ref{u22}, the decomposition $\vec{h}_0= \vec{h}-g\vec{H}$ and the estimates \eqref{est12} finishes the proof.

\end{proof}

\noindent
Our goal is to estimate the terms on the right hand side of inequality \eqref{noncy} of Lemma \ref{noncy1} with regards to the $L^2$ and $L^4$ norms of the second fundamental form. We will use the following Michael-Sobolev inequality in order to estimate these terms.   As noted in \cite {xavier1}, the following statement does not depend on the topology of the manifold $M$ .

\begin{theo}[\cite{allard, xavier1, xavier, hoffman, michael, glen}] \label{sobolevineq}
Let $\vec{\Phi}: M \rightarrow \mathbb{R}^n$ be a smooth immersion of the $m$-dimensional manifold $M$ with boundary $\partial M$ into $\mathbb{R}^n$. Then for any $u\in C^1(\overline{M})$, there exists a positive constant $K$, depending on $m$ and $r$ such that 
\begin{align}
\left(\int_M |u|^{\frac{mp}{m-r}} d\mu\right)^{\frac{m-r}{m}} \leq K\int_M (|\nabla u|^r + |Hu|^r) \, d\mu + K \int_{\partial M} |u|^r d\sigma. \label{sob}
\end{align} 
\end{theo}
\noindent
If $M$ is minimal with no boundary, then one obtains the  Sobolev embedding: $W^{1,p}\subseteq L^{p^\star}$, where $p^\star$ is the Sobolev conjugate.  Note that we have added a boundary term in \eqref{sob} because we are working on manifolds with boundary. In our setting, the dimension $m$ is 4 and we may choose $r=2$ so that inequality \eqref{sob} reduces to
\begin{align}
\left(\int_M |u|^4 d\mu\right)^{\frac{1}{2}} \leq K\int_M (|\nabla u|^2 + |Hu|^2) \, d\mu + K \int_{\partial M} |u|^2 d\sigma. \nonumber
\end{align}

\noindent
Our next lemma presents the estimate of the non $c_\gamma$ terms on the right hand side of the inequality \eqref{noncy} in Lemma \ref{noncy1}. Bearing in mind the inequalities in \eqref{est12} and the decomposition $\vec{h}_0=\vec{h}-g\vec{H}$, one verifies that $|\vec{h}_0|^2|\pi_{\vec{n}}\nabla\vec{h}_0|^2 \leq c|\vec{h}|^2|\pi_{\vec{n}}\nabla\vec{h}|^2$ and so we prove the following.
\begin{lem} \label{b2}
Let $\vec{\Phi}:\Sigma\rightarrow \mathbb{R}^m$ be an immersion of a 4-dimensional manifold $\Sigma$ whose boundaries satisfy the condition \eqref{bdry} and let $\gamma$ be the cut-off function in \eqref{cutoff} then 
\begin{align}
&\int_\Sigma |\vec{h}|^2 |\pi_{\vec{n}}\nabla\vec{h}|^2 \gamma^p d\mu +\int_\Sigma |\vec{h}|^4 \gamma^p d\mu  +\int_\Sigma |\vec{h}|^6\gamma^p d\mu   \nonumber
\\& \leq c\|\vec{h}\|_{L^4,[\gamma>0]}^2\left[\int_\Sigma |\pi_{\vec{n}}\nabla_{(2)}\vec{h}|^2 \gamma^p d\mu   +  c_\gamma^4\int_\Sigma |\vec{h}|^2 \gamma^{p-4} d\mu  +\delta \int_\Sigma |\pi_{\vec{n}}\nabla_{(2)}\vec{h}|^2 \gamma^p d\mu         \right]\nonumber
\\&\quad + c\|\vec{h}\|_{L^4,[\gamma>0]}^2\left[ c_\gamma^2\int_\Sigma |\vec{h}|^4 \gamma^{p-4} d\mu + \int_\Sigma |\vec{h}_0|^2 |\vec{h}|^4 \gamma^p d\mu     \right].\nonumber
\end{align}
\end{lem}

\noindent
\begin{proof}
We apply the Cauchy-Schwartz estimate and choose $u= |\vec{h}|^2 \gamma^{p/2}$ in Theorem \ref{sobolevineq} to have
\begin{align}
\int_\Sigma |\vec{h}|^4 \gamma^p d\mu  +\int_\Sigma |\vec{h}|^6 d\mu &\leq c\int_\Sigma |\vec{h}|^4 \gamma^p d\mu + \|\vec{h}\|_{L^4, [\gamma> 0]}^2 \left(\int_\Sigma |\vec{h}|^8 \gamma^{2p} d\mu \right)^{1/2}  \nonumber
\\& \leq c\int_\Sigma |\vec{h}|^4 \gamma^p d\mu + c\|\vec{h}\|_{L^4,[\gamma>0]}^2 \left[ \int_\Sigma |\vec{h}|^2|\pi_{\vec{n}}\nabla\vec{h}|^2 \gamma^p d\mu +  \int_\Sigma |\vec{h}|^4 \gamma^p d\mu \right. \nonumber
\\&\quad\quad \left.+ c_\gamma^2\int_\Sigma |\vec{h}|^4 \gamma^{p-4} d\mu  + \int_\Sigma |\vec{h}_0|^2|\vec{h}|^4 \gamma^p d\mu \right] . \label{bn1}
\end{align}

\noindent
 In a similar fashion, the Cauchy-Schwartz estimate and the Michael-Sobolev inequality (applied to $u=|\pi_{\vec{n}}\nabla\vec{h}|\gamma^{p/2}$ ) yield
\begin{align}
\int_\Sigma |\vec{h}|^2 |\pi_{\vec{n}}\nabla\vec{h}|^2 \gamma^p d\mu &\leq \|\vec{h}\|_{L^4,[\gamma >0]}^2 \left(\int_\Sigma|\pi_{\vec{n}}\nabla\vec{h}|^4 \gamma^{2p}  d\mu \right)^{1/2}             \nonumber
\\& \leq c\|\vec{h}\|_{L^4,[\gamma >0]}^2 \int_\Sigma ( |\pi_{\vec{n}}\nabla_{(2)}\vec{h}|^2 \gamma^p + c_\gamma^2 |\pi_{\vec{n}}\nabla\vec{h}|^2 \gamma^{p-2} +|\vec{h}|^2 |\pi_{\vec{n}}\nabla \vec{h}|^2 \gamma^p      ) d\mu \label{middle}
\end{align}
and we can estimate the middle term on the right hand side of \eqref{middle} via divergence theorem and the Peter-Paul inequality. We have for $\delta >0$,
\begin{align}
cc_\gamma^2 \int_\Sigma |\pi_{\vec{n}}\nabla\vec{h}|^2 \gamma^{p-2} d\mu &\leq \delta \int_\Sigma |\pi_{\vec{n}}\nabla_{(2)}\vec{h}|^2  \gamma^p d\mu + c(\delta)c_\gamma^4 \int_\Sigma |\vec{h}|^2 \gamma^{p-4} d\mu. \label{jk18}
\end{align}
The proof is finished by using \eqref{jk18} in \eqref{middle}, combining the result with \eqref{bn1} and absorbing.

\noindent
Our next task is to estimate the $c_\gamma^2$ term and the $U^i$ term on the right hand side of the inequality \eqref{noncy} in Lemma \ref{noncy1}. This is presented in the following lemma.
\begin{lem} \label{b1}
Let $\vec{\Phi}:\Sigma\rightarrow \mathbb{R}^m$ be an immersion of a 4-dimensional manifold $\Sigma$ whose boundaries satisfy the condition \eqref{bdry} and let $\gamma$ be the cut-off function in \eqref{cutoff} then 
\begin{align}
&cc_\gamma^2 \int_\Sigma |\pi_{\vec{n}}\nabla\vec{h}_0|^2 \gamma^{p-2} d\mu +2p\int_\Sigma U^i \nabla_i\gamma \gamma^{p-1} d\mu \nonumber
\\&\leq \delta \int_\Sigma |\pi_{\vec{n}}\nabla_{(2)}\nabla\vec{h}|^2 \gamma^p d\mu +\delta \int_\Sigma |\vec{h}|^2 |\pi_{\vec{n}}\nabla\vec{h}|^2 \gamma^p d\mu + c(\delta)c_\gamma^4 \int_\Sigma |\vec{h}|^2 \gamma^{p-4} d\mu.   \nonumber
\end{align}
\end{lem}

\begin{proof}
We handle the $c_\gamma^2$ term by observing that $|\pi_{\vec{n}}\nabla\vec{h}_0|^2\leq c|\pi_{\vec{n}}\nabla\vec{h}|^2$ and using inequality \eqref{jk18}. We have 
\begin{align}
cc_\gamma^2 \int_\Sigma |\pi_{\vec{n}}\nabla\vec{h}_0|^2 \gamma^{p-2} d\mu &\leq \delta \int_\Sigma |\pi_{\vec{n}}\nabla_{(2)}\vec{h}|^2  \gamma^p d\mu + c(\delta)c_\gamma^4 \int_\Sigma |\vec{h}|^2 \gamma^{p-4} d\mu. \label{p1}
\end{align}
Next, we recall the expression for $U$. We have
\begin{align}
U^i&:=-\frac{1}{2}\Delta_\perp \vec{H}\cdot\nabla^i\vec{H} +\frac{1}{2} \vec{H}\cdot\nabla^i\Delta_\perp \vec{H} + J^i  \nonumber
\end{align}
where, for convenience, we have set 
\begin{align}
J^i&= \frac{1}{2}(\nabla^i\vec{H}\cdot\vec{h}_{jk})(\vec{H}\cdot\vec{h}^{jk})   -2(\nabla^j\vec{H}\cdot\vec{h}^i_j)|\vec{H}|^2 +\nabla^j|\vec{H}|^2(\vec{H}\cdot\vec{h}^i_j) \nonumber
\\& \quad + \frac{1}{2}\nabla^i((\vec{H}\cdot\vec{h}_{kj})\vec{h}^{kj}) \cdot\vec{H} +2 \nabla_j((\vec{H}\cdot\vec{h}^{ij})\vec{H})\cdot\vec{H}
-7\nabla^i(|\vec{H}|^2\vec{H})\cdot\vec{H}.  \nonumber
\end{align}
Using \eqref{jk18}, we find the estimate
\begin{align}
-\int_\Sigma\Delta_\perp \vec{H}\cdot\nabla^i\vec{H} \gamma^{p-1} \nabla_i \gamma d\mu \leq \delta \int_\Sigma |\pi_{\vec{n}}\nabla_{(2)}\vec{h}|^2  \gamma^p d\mu + c(\delta)c_\gamma^4 \int_\Sigma |\vec{h}|^2 \gamma^{p-4} d\mu. \label{p2}
\end{align}

\noindent
We note that
\begin{align}
\int_\Sigma\vec{H}\cdot\pi_{\vec{n}}\nabla\Delta_\perp  \vec{H} \gamma^{p-1} \nabla\gamma d\mu \leq c c_\gamma\int_\Sigma \vec{H}\cdot\pi_{\vec{n}}\nabla\Delta_\perp  \vec{H} \gamma^{p-1}  d\mu              \label{dfg1}
\end{align}
We use divergence theorem on the left hand side of \eqref{dfg1} to obtain the next estimate 
\begin{align}
\int_\Sigma\vec{H}\cdot\pi_{\vec{n}}\nabla\Delta_\perp  \vec{H} \gamma^{p-1} \nabla\gamma d\mu &\leq \delta\int_\Sigma |\pi_{\vec{n}}\nabla_{(2)}\vec{h}|^2 \gamma^p d\mu + cc_\gamma^2 \int_\Sigma |\pi_{\vec{n}}\nabla\vec{h}|^2 \gamma^{p-2} d\mu +c c_\gamma^4 \int_\Sigma |\vec{h}|^2 \gamma^{p-4} d\mu  \nonumber
\\& \overset{\eqref{jk18}}\leq  \delta\int_\Sigma |\pi_{\vec{n}}\nabla_{(2)}\vec{h}|^2 \gamma^p d\mu +c(\delta) c_\gamma^4 \int_\Sigma |\vec{h}|^2 \gamma^{p-4} d\mu.  \label{p3}
\end{align}
With the help of the Peter-Paul inequality we estimate 
\begin{align}
&\int_\Sigma J^i \nabla_i \gamma \gamma^{p-1} d\mu  \leq \delta \int_\Sigma |\vec{h}|^2 |\pi_{\vec{n}}\nabla\vec{h}|^2 \gamma^p d\mu +   \delta  \int_\Sigma |\vec{h}|^4 \gamma^p d\mu +              c(\delta) c_\gamma^4 \int_\Sigma |\vec{h}|^4 \gamma^{p-4} d\mu .\label{p4} 
\end{align}
The result then follows from the estimates \eqref{p1}, \eqref{p2}, \eqref{p3} and \eqref{p4}.

\end{proof}

\noindent
\begin{lem} \label{con1}
Let $\vec{\Phi}:\Sigma\rightarrow \mathbb{R}^m$ be an immersion of a 4-dimensional manifold $\Sigma$ whose boundaries satisfy the condition \eqref{bdry} and let $\gamma$ be the cut-off function in \eqref{cutoff} then 
\begin{align}
&\int_\Sigma (|\pi_{\vec{n}}\nabla_{(2)}\vec{h}|^2 +|\vec{h}|^2|\pi_{\vec{n}}\nabla\vec{h}|^2+|\vec{h}|^4|\vec{h}_0|^2  ) \gamma^p d\mu                                    \nonumber
\\& \leq c\int_\Sigma |\mathcal{\vec{W}}\cdot \vec{H}| \gamma^p d\mu+ cc_\gamma^4 \int_\Sigma |\vec{h}|^2 \gamma^{p-4} d\mu   \nonumber
 + cc_\gamma^4 \int_\Sigma |\vec{h}|^4 \gamma^{p-4} d\mu  \nonumber
\\& \quad+c\|\vec{h}\|_{L^4,[\gamma>0]}^2\left[   c_\gamma^4\int_\Sigma |\vec{h}|^2 \gamma^{p-4} d\mu +c_\gamma^2\int_\Sigma |\vec{h}|^4 \gamma^{p-4} d\mu        \right]\label{final}
\end{align}
\end{lem}

\noindent
\begin{proof}
The result is obtained by using the inequalities of Lemmas \ref{b2} and \ref{b1} in Lemma \ref{noncy1}.
\end{proof}

\noindent
Having obtained the estimates in Lemma \ref{con1}, we are now ready to conclude the proof of  Theorem \ref{mainr}. 
\vskip3mm
\noindent
{\bf Proof of Theorem \ref{mainr}}
The critical points of  $\mathcal{E}(\Sigma)$ satisfy $\mathcal{\vec{W}}=\vec{0}.$ We choose $p=4$ so that the left hand side of inequality \eqref{final} is bounded above with regards to the $L^2$ and $L^4$-norms of the second fundamental form which by assumption are bounded by $\varepsilon$. We then take $\rho \rightarrow \infty$ so that $c_\gamma \rightarrow 0$ and hence $|\vec{h}_0|^2=0$. Bearing in mind the boundary condition \eqref{bdry} we conclude that $\Sigma$ is an umbilic Willmore submanifold with totally geodesic boundary.

\end{proof}

\appendix
\section{Miscellaneous Computations}
Recall that the four dimensional Willmore invariant $\mathcal{\vec{W}}$ is a certain linear combination of the quantities $\mathcal{\vec{W}}_1$, $\mathcal{\vec{W}}_2$ and $\mathcal{\vec{W}}_3$ which were defined in Section \ref{will67}. The following lemma is a curvature identity  involving $\mathcal{\vec{W}}_1$. 
\begin{lem}\label{firstlem}
Let $\vec{\Phi}:\Sigma\rightarrow \mathbb{R}^m$ be an immersion of a 4-dimensional manifold $\Sigma$ whose boundaries satisfy the condition \eqref{bdry} and let $\gamma$ be the cut-off function in \eqref{cutoff} then 
\begin{align}
\frac{1}{2}\int_\Sigma|\Delta_\perp\vec{H}|^2 \gamma^p d\mu &=-\int_\Sigma \mathcal{\vec{W}}_1\cdot \vec{H} \gamma^p d\mu +\frac{1}{2} \int_\Sigma \pi_{\vec{n}}\nabla^j\vec{H}\cdot\nabla_j[(\vec{H}\cdot\vec{h}^{ik})\vec{h}_{ik}]
\gamma^p d\mu -2\int_\Sigma (\nabla^j\vec{H}\cdot\vec{h}^k_j)(\vec{H}\cdot\nabla_k\vec{H})\gamma^p d\mu  \nonumber
\\&\quad +2 \int_\Sigma (\vec{H}\cdot\vec{h}^k_j)(\pi_{\vec{n}}\nabla^j\vec{H}\cdot\nabla_k\vec{H})\gamma^p d\mu  +\int_\Sigma\left(2(\pi_{\vec{n}}\nabla_i\vec{H}\cdot\pi_{\vec{n}}\nabla_j\vec{H})\vec{h}^{ij}
-4|\pi_{\vec{n}}\nabla\vec{H}|^2\vec{H}\right)\cdot\vec{H} \gamma^p d\mu  \nonumber
\\&\quad +p\int_\Sigma \left(-\frac{1}{2}\Delta_\perp \vec{H}\cdot \nabla^i\vec{H}  +\frac{1}{2} \nabla^i\Delta_\perp \vec{H}\cdot\vec{H} +\frac{1}{2}(\nabla^i\vec{H}\cdot\vec{h}_{jk})(\vec{H}\cdot\vec{h}^{jk})  \right. \nonumber
\\&\quad\quad\quad\left. -2(\nabla^j\vec{H}\cdot\vec{h}^i_j)|\vec{H}|^2 +2(\vec{H}\cdot\vec{h}^i_j)(\vec{H}\cdot\nabla^j\vec{H})   \right)  \gamma^{p-1} \nabla_i\gamma d\mu     \label{fisg}
\end{align}
where $p>0$ is a constant.
\end{lem}
\noindent
\begin{proof}
We integrate by parts and use the boundary condition \eqref{bdry} to obtain
\begin{align}
\int_{\Sigma}  \Delta_\perp^2\vec{H}\cdot\vec{H}\gamma^p d\mu&=  \int_{\Sigma} |\Delta_\perp \vec{H}|^2 \gamma^p d\mu +p\int_{\Sigma} (\Delta_\perp\vec{H}\cdot\nabla_i\vec{H}-  \nabla^i\Delta_\perp\vec{H}\cdot\vec{H}) \gamma^{p-1} \nabla_i\gamma  d\mu         \label{one}
\end{align}
and 
\begin{align}
\int_{\Sigma}( \Delta_\perp\vec{H}\cdot\vec{h}_{ik}) (\vec{H}\cdot\vec{h}^{ik}) \gamma^p d\mu &=   -\int_{\Sigma} \pi_{\vec{n}}\nabla^j\vec{H}\cdot\nabla_j[(\vec{H}\cdot\vec{h}^{ik})\vec{h}_{ik}] \gamma^p d\mu     -p\int_{\Sigma} (\pi_{\vec{n}}\nabla^j\vec{H}\cdot\vec{h}_{ik} )(\vec{H}\cdot\vec{h}^{ik})\gamma^{p-1}\nabla_j\gamma  d\mu.              \label{two}
\end{align}
Applying divergence theorem and the boundary condition \eqref{bdry} we have
\begin{align}
\int_{\Sigma} \pi_{\vec{n}}\nabla_k[(\nabla^j\vec{H}\cdot\vec{h}^k_j)\vec{H}]
\cdot\vec{H} \gamma^p d\mu&= -\int_{\Sigma} (\nabla^j\vec{H}\cdot\vec{h}^k_j)(\vec{H}\cdot\nabla_k\vec{H}) \gamma^p d\mu -p\int_\Sigma \nabla^j\vec{H}\cdot\vec{h}^k_j |\vec{H}|^2 \gamma^{p-1}\nabla_k \gamma  d\mu    \label{three}
\end{align}
and
\begin{align}
\int_\Sigma \pi_{\vec{n}} \nabla_k[(\vec{H}\cdot\vec{h}^k_j)\pi_{\vec{n}}\nabla^j\vec{H}]\cdot\vec{H}  
\gamma^p d\mu &=  -\int_\Sigma (\vec{H}\cdot\vec{h}^k_j)(\pi_{\vec{n}}\nabla^j\vec{H}\cdot\nabla_k\vec{H} ) \gamma^p d\mu      -p\int_{\Sigma} (\vec{H}\cdot\vec{h}_j^k)(\vec{H}\cdot\nabla^j\vec{H})\gamma^{p-1}\nabla_k\gamma  d\mu . \label{four}
\end{align}

\noindent
We multiply $\mathcal{\vec{W}}_1$ by $\vec{H}\gamma^p$ for some constant $p>0$, to have
\begin{align}
\int_{\Sigma}\mathcal{\vec{W}}_1\cdot\vec{H} \gamma^p d\mu &=    -\frac{1}{2} \int_\Sigma\Delta_\perp ^2\vec{H}   \cdot\vec{H} \gamma^p d\mu     
-\frac{1}{2} \int_{\Sigma}(\Delta_\perp \vec{H}\cdot\vec{h}_{ik}) (\vec{h}^{ik}\cdot\vec{H}) \gamma^p d\mu +2\int_\Sigma\pi_{\vec{n}} \nabla_k((\nabla^j\vec{H}\cdot\vec{h}_j^k)\vec{H})\cdot\vec{H} \gamma^p d\mu   \nonumber
\\& \quad-2\int_\Sigma\pi_{\vec{n}}\nabla_k((\vec{H}\cdot \vec{h}^k_j)\pi_{\vec{n}}\nabla^j\vec{H}) \cdot\vec{H} \gamma^p d\mu +\int_\Sigma\left(2(\pi_{\vec{n}}\nabla_i\vec{H}\cdot\pi_{\vec{n}}\nabla_j\vec{H})\vec{h}^{ij}
-4|\pi_{\vec{n}}\nabla\vec{H}|^2\vec{H}\right)\cdot\vec{H} \gamma^p d\mu  \label{meq}
\end{align}

\noindent
We can then obtain  \eqref{fisg} by substituting \eqref{one}, \eqref{two}, \eqref{three} and \eqref{four} into \eqref{meq}. 
\end{proof}

\noindent
We obtain a similar curvature identity involving $\mathcal{\vec{W}}_2$.
\begin{lem} \label{seclem}
Let $\vec{\Phi}:\Sigma\rightarrow \mathbb{R}^m$ be an immersion of a 4-dimensional manifold $\Sigma$ whose boundaries satisfy the condition \eqref{bdry} and let $\gamma$ be the cut-off function in \eqref{cutoff} then 
\begin{align}
\int_\Sigma \mathcal{\vec{W}}_2\cdot\vec{H} \gamma^p d\mu  &= -\frac{1}{2}\int_\Sigma\nabla^k((\vec{H}\cdot\vec{h}^{ij})\vec{h}_{ij})\cdot \pi_{\vec{n}}\nabla_k\vec{H} \gamma^p d\mu   -2\int_\Sigma\nabla_j((\vec{H}\cdot\vec{h}^{ij})\vec{H})\cdot\nabla_i\vec{H}
\gamma^p d\mu   \nonumber
\\&\quad\quad+ 4\int_\Sigma(\vec{H}\cdot\vec{h}_{ij})(\vec{H}\cdot\vec{h}^i_k) (\vec{H}\cdot\vec{h}^{jk})\gamma^p d\mu-4\int_\Sigma |\vec{H}\cdot\vec{h}|^2|\vec{H}|^2\gamma^p d\mu  \nonumber
\\&\quad\quad +\frac{1}{2}\int_\Sigma(\vec{H}\cdot\vec{h}^{ij})(\vec{h}_{kl}\cdot\vec{h}_{ij})( \vec{H}\cdot\vec{h}^{kl} ) \gamma^p d\mu \nonumber       
\\&-p\int_\Sigma \left(\frac{1}{2}\nabla^k((\vec{H}\cdot\vec{h}^{ij})\vec{h}_{ij})\cdot\vec{H}    +2 \nabla_j((\vec{H}\cdot\vec{h}^{kj})\vec{H})\cdot\vec{H} \right)\gamma^{p-1}\nabla_k\gamma d\mu  \label{sdfa2}
\end{align}
where $p>0$ is a constant.
\end{lem}

\noindent
\begin{proof}
To prove \eqref{sdfa2}, we first use integration by parts to find
\begin{align}
\int_\Sigma \Delta_\perp [(\vec{H}\cdot\vec{h}^{ij})\vec{h}_{ij}]\cdot\vec{H} \gamma^p d\mu
&= \int_\Sigma  \nabla^k\left(\nabla_k((\vec{H}\cdot\vec{h}^{ij})\vec{h}_{ij})  \cdot\vec{H} \gamma^p\right) d\mu  -\int_{\Sigma} \nabla^k((\vec{H}\cdot\vec{h}^{ij})\vec{h}_{ij})\cdot\pi_{\vec{n}}\nabla_k\vec{H}
\gamma^p d\mu  \nonumber
\\&\quad\quad-p\int_\Sigma \nabla^k((\vec{H}\cdot\vec{h}^{ij})\vec{h}_{ij})\cdot\vec{H}\gamma^{p-1} \nabla_k\gamma d\mu.    \label{vad}
\end{align}
Observe that Codazzi equation yields
\begin{align}
\vec{H}\cdot\nabla^k((\vec{H}\cdot\vec{h}^{ij})\vec{h}_{ij})= (\nabla^k\vec{H}\cdot\vec{h}^{ij})(\vec{H}\cdot\vec{h}_{ij})+ (\vec{H}\cdot\nabla^i\vec{h}^{kj})(\vec{H}\cdot\vec{h}_{ij}) +(\vec{H}\cdot\vec{h}^{ij})(\vec{H}\cdot\nabla_i\vec{h}_{j}^k)  \label{vad1}
\end{align}
Substituting \eqref{vad1} into \eqref{vad}, applying divergence theorem and the condition \eqref{bdry} guarantees that the first term of \eqref{vad} vanishes.
In a similar fashion, one finds
\begin{align}
\int_\Sigma \pi_{\vec{n}}\nabla_i\nabla_j((\vec{H}\cdot\vec{h}^{ij})\vec{H})\cdot\vec{H} 
\gamma^p d\mu &= -\int_\Sigma \nabla_j((\vec{H}\cdot\vec{h}^{ij})\vec{H})\cdot \nabla_i\vec{H} \gamma^p d\mu   -p\int_\Sigma \nabla_j((\vec{H}\cdot\vec{h}^{ij})\vec{H})\cdot\vec{H} \nabla_i\gamma \gamma^{p-1}   \nonumber
\end{align}

\noindent
Integrating the dot product of $\mathcal{\vec{W}}_2$ and $\vec{H}\gamma^p$ yields
\begin{align}
\int_\Sigma \mathcal{\vec{W}}_2\cdot \vec{H}\gamma^p d\mu
&=     \frac{1}{2}\int_\Sigma\Delta_\perp ((\vec{H}\cdot\vec{h}^{ij})\vec{h}_{ij}) \cdot \vec{H}\gamma^p d\mu +2\int_\Sigma\pi_{\vec{n}}\nabla_i\nabla_j ((\vec{H}\cdot\vec{h}^{ij})\vec{H}) \cdot \vec{H}\gamma^p d\mu \nonumber
\\&\quad\quad+ 4\int_\Sigma(\vec{H}\cdot\vec{h}_{ij})(\vec{H}\cdot\vec{h}^i_k) (\vec{H}\cdot\vec{h}^{jk})\gamma^p d\mu-4\int_\Sigma |\vec{H}\cdot\vec{h}|^2|\vec{H}|^2\gamma^p d\mu  \nonumber
\\&\quad\quad +\frac{1}{2}\int_\Sigma(\vec{H}\cdot\vec{h}^{ij})(\vec{h}_{kl}\cdot\vec{h}_{ij})( \vec{H}\cdot\vec{h}^{kl} ) \gamma^p d\mu \nonumber
\end{align}

\end{proof}

\noindent
We obtain yet another curvature identity involving $\mathcal{\vec{W}}_3$.
\begin{lem} \label{thirdlem}
Let $\vec{\Phi}:\Sigma\rightarrow \mathbb{R}^m$ be an immersion of a 4-dimensional manifold $\Sigma$ whose boundaries satisfy the condition \eqref{bdry} and let $\gamma$ be the cut-off function in \eqref{cutoff} then 
\begin{align}
\int_\Sigma \mathcal{\vec{W}}_3\cdot\vec{H} \gamma^p d\mu
&=-\int_\Sigma \nabla^j(|\vec{H}|^2\vec{H})\cdot\pi_{\vec{n}}\nabla_j\vec{H} \gamma^p d\mu +\int_\Sigma|\vec{H}|^2|\vec{H}\cdot\vec{h}|^2\gamma^p d\mu -4\int_\Sigma|\vec{H}|^6 \gamma^p d\mu \nonumber
\\&\quad\quad  -p\int_\Sigma \nabla^j(|\vec{H}|^2\vec{H})\cdot\vec{H} \gamma^{p-1}\nabla_j\gamma  d\mu.   \label{sdfa22}
\end{align}
where $p>0$ is a constant.
\end{lem}

\noindent
\begin{proof}
Identity \eqref{sdfa22} easily follows from 
\begin{align}
\int_\Sigma \Delta_\perp (|\vec{H}|^2\vec{H})\cdot\vec{H} \gamma^p d\mu =-\int_\Sigma \nabla^j(|\vec{H}|^2\vec{H})\cdot\pi_{\vec{n}}\nabla_j\vec{H} \gamma^p d\mu  -p\int_\Sigma \nabla^j(|\vec{H}|^2\vec{H})\cdot\vec{H}\gamma^{p-1} \nabla_j\gamma  d\mu  \nonumber
\end{align}
where we have used integration by parts, divergence theorem and \eqref{bdry}.
\end{proof}

\subsection{Interchaning Covariant Derivatives} \label{interchange}
By splitting into normal and tangential components, we compute
\begin{align}
g^{ij}\nabla_i\pi_{\vec{n}} \nabla_j\pi_{\vec{n}}\nabla_k\vec{H} &=g^{ij} \nabla_i\nabla_j\pi_{\vec{n}} \nabla_k\vec{H} - \nabla^j\pi_T\nabla_j\pi_{\vec{n}} \nabla_k\vec{H}        \nonumber
\\&= g^{ij} \nabla_i\nabla_j \nabla_k\vec{H} - \Delta\pi_{T} \nabla_k\vec{H} -\nabla^j\pi_T\nabla_j\pi_{\vec{n}} \nabla_k\vec{H} \nonumber   
\\&=   g^{ij} \nabla_i\nabla_k \nabla_j\vec{H} - \Delta\pi_{T} \nabla_k\vec{H} -\nabla^j\pi_T\nabla_j\pi_{\vec{n}} \nabla_k\vec{H}                   \nonumber
\\&=    g^{ij} \nabla_k\nabla_i \nabla_j\vec{H} +g^{ij} g^{el}R_{ikjl}\nabla_e\vec{H}- \Delta\pi_{T} \nabla_k\vec{H} -\nabla^j\pi_T\nabla_j\pi_{\vec{n}} \nabla_k\vec{H}                   \nonumber  
\\& =  \nabla_k\Delta\vec{H} +g^{ij} g^{el}R_{ikjl}\nabla_e\vec{H}- \Delta\pi_{T} \nabla_k\vec{H} -\nabla^j\pi_T\nabla_j\pi_{\vec{n}} \nabla_k\vec{H} .                                    \nonumber        
\end{align}

\noindent
By splitting into tangential and normal components, we compute
\begin{align}
\Delta_\perp\pi_{\vec{n}} \nabla_j \vec{h}_{kl}&=g^{ia}\pi_{\vec{n}} \nabla_i\pi_{\vec{n}} \nabla_a \pi_{\vec{n}} \nabla_j \vec{h}_{kl} = g^{ia}\pi_{\vec{n}} \nabla_i\pi_{\vec{n}} \nabla_a  \nabla_j \vec{h}_{kl}  - \pi_{\vec{n}} \nabla_i\pi_{\vec{n}} \nabla^i \pi_{T} \nabla_j \vec{h}_{kl}   \nonumber
\\& = g^{ia}\pi_{\vec{n}} \nabla_i \nabla_a  \nabla_j \vec{h}_{kl}- \pi_{\vec{n}} \nabla_i\pi_{T} \nabla^i \nabla_j \vec{h}_{kl}  - \pi_{\vec{n}} \nabla_i\pi_{\vec{n}} \nabla^i \pi_{T} \nabla_j \vec{h}_{kl}   \nonumber
\\& =     g^{ia} \nabla_i \nabla_a  \nabla_j \vec{h}_{kl}- \pi_{T} \Delta  \nabla_j \vec{h}_{kl}    - \pi_{\vec{n}} \nabla_i\pi_{T} \nabla^i \nabla_j \vec{h}_{kl}  - \pi_{\vec{n}} \nabla_i\pi_{\vec{n}} \nabla^i \pi_{T} \nabla_j \vec{h}_{kl}   \nonumber               
\\& = g^{ia}\nabla_j \nabla_i  \nabla_a \vec{h}_{kl} + (\vec{h}* \pi_{\vec{n}}\nabla\vec{h} * \vec{h} + \vec{h}*  \vec{h}*\pi_{\vec{n}}\nabla\vec{h} )_{jkl}   \nonumber
\\& \quad \quad - \pi_{T} \Delta  \nabla_j \vec{h}_{kl}    - \pi_{\vec{n}} \nabla_i\pi_{T} \nabla^i \nabla_j \vec{h}_{kl}  - \pi_{\vec{n}} \nabla_i\pi_{\vec{n}} \nabla^i \pi_{T} \nabla_j \vec{h}_{kl}   \nonumber       
\\&   =      g^{ia}\nabla_j \nabla_i \pi_{\vec{n}} \nabla_a \vec{h}_{kl}  +  \nabla_j \nabla_i \pi_{T} \nabla^i \vec{h}_{kl}+ (\vec{h}*\pi_{\vec{n}} \nabla\vec{h} * \vec{h} + \vec{h}*  \vec{h}*\pi_{\vec{n}}\nabla\vec{h} )_{jkl}   \nonumber
\\& \quad \quad - \pi_{T} \Delta  \nabla_j \vec{h}_{kl}    - \pi_{\vec{n}} \nabla_i\pi_{T} \nabla^i \nabla_j \vec{h}_{kl}  - \pi_{\vec{n}} \nabla_i\pi_{\vec{n}} \nabla^i \pi_{T} \nabla_j \vec{h}_{kl}   \nonumber   
\\&    =   \pi_{\vec{n}} \nabla_j \Delta_\perp \vec{h}_{kl} + \pi_{T} \nabla_j \Delta_\perp \vec{h}_{kl} +\nabla_j \pi_T \nabla_i \pi_{\vec{n}} \nabla^i \vec{h}_{kl}      +  \nabla_j \nabla_i \pi_{T} \nabla^i \vec{h}_{kl}  \nonumber
\\&   \quad\quad+ (\vec{h}* \pi_{\vec{n}}\nabla\vec{h} * \vec{h} + \vec{h}*  \vec{h}*\pi_{\vec{n}}\nabla\vec{h} )_{jkl}   \nonumber
\\& \quad \quad - \pi_{T} \Delta  \nabla_j \vec{h}_{kl}    - \pi_{\vec{n}} \nabla_i\pi_{T} \nabla^i \nabla_j \vec{h}_{kl}  - \pi_{\vec{n}} \nabla_i\pi_{\vec{n}} \nabla^i \pi_{T} \nabla_j \vec{h}_{kl}.  \nonumber
\end{align}


\begin{thebibliography}{50}

\bibitem{allard} W. K. Allard, On the first variation of a varifold, {\it Ann. of Math.} {\bf 95}(2): 417–491, 1972.

\bibitem{ancari} S. Ancari and I. Miranda, Rigidity theorems for complete $\lambda$-hypersurfaces, {\it Archiv der Mathematik}, {\bf 117}, 105-120, 2021. 

\bibitem{anderson} M.T. Anderson, Convergence and rigidity of manifolds under Ricci curvature bounds, {\it Invent. Math.} {\bf 102}(2): 429-445, 1990. 

\bibitem{xavier1} X. Cabr\'e, M. Cozzi and G. Csat\'o, A fractional Michael-Sobolev inequality on convex hypersurfaces,   {\it Ann. Inst. H. Poincar\'e Anal. Non Lin\'eaire }, {\bf 40} :185-214, 2022, DOI 10.4171/AIHPC/39.

\bibitem{xavier} X. Cabr\'e and P. Miraglio, Universal Hardy-Sobolev inequalities on hypersurfaces of Euclidean space, {\it Communications in Contemporary Mathematics}, {\bf 24}(05): 25 pages, 2022, doi.org/10.1142/S0219199721500632. 


\bibitem{gqchen}G-Q.G. Chen, S. Li, Global weak rigidity of the Gauss–Codazzi–Ricci equations and isometric immersions of Riemannian manifolds with lower regularity, {\it J. Geom. Anal.}, {\bf28}:1957–2007, 2018, doi.org/10.1007/s12220-017-9893-1.


\bibitem{cheng} Q-M. Cheng, S. Ogata, G. Wei,  Rigidity theorems of $\lambda$-hypersurfaces. {\it Comm. Anal. Geom}. {\bf24}(1), 45-58, 2016.

\bibitem{chern} S.S. Chern, M. do Carmo, S. Kobayashi, Minimal submanifolds of a sphere with second fundamental
form of constant length. In: {\it Proceedings of a Conference for M. Stone, University of Chicago, Chicago,
III., 1968, Functional Analysis and Related Fields,} pp. 59–75. Springer, New York, 1970.




\bibitem{dorel} D. Fetcu, E. Loubeau, and C. Oniciuc, Bochner–Simons Formulas and the Rigidity of Biharmonic Submanifolds, {\it The Journal of Geometric Analysis},  {\bf 31}, 1732–1755, 2021.

\bibitem{coibre} D. Fischer-Colbrie, Some rigidity theorems for minimal submanifolds of the sphere,{\it Acta Math.} {\bf145}(1–
2), 29-46, 1980.

\bibitem{robingraham} C. R. Graham and N. Reichert, Higher-dimensional Willmore energies via minimal submanifold asymptotics, {\it Asian Journal of Mathematics} {\bf24}(4): 571-610, 2020.

\bibitem{guven} J. Guven, Conformally invariant bending energy for hypersurfaces, {\it J. Phys. A: Math. Gen.} {\bf38}: 7943–7955, 2005.

\bibitem{hoffman}  D. Hoffman and J. Spruck, Sobolev and isoperimetric inequalities for Riemanniansubmanifolds, {\it Commun. Pure Appl. Math.}{\bf 27}: 715–727, 1974.

\bibitem{kuwert} E. Kuwert and R. Sch\"atzle, Gradient flow for the Willmore functional, {\it Communications in Contemporary Mathematics}, {\bf 10}(02): 307-339, 2002.

\bibitem{tobias} T. Lamm, H. T. Nguyen, Quantitative regidity results for conformal immersions, {\it American Journal of Mathematics}, {\bf 136}, (5):  1409-1440, 2014.


\bibitem{tobiasl} T. Lamm, R.M. Sch\"atzle,  Optimal rigidity estimates for nearly umbilical surfaces in arbitrary codimension, {\it Geometric and Functional Analysis}, {\bf 24}, 2029–2062, 2014. 


\bibitem{lawson} H.B. Lawson Jr.,  Local rigidity theorems for minimal hypersurfaces. {\it Ann. of Math.} (2) 89, 187–197,
1969.


\bibitem{mmcoy} J. McCoy, G. Wheeler, A rigidity theorem for ideal surfaces with flat boundary, {\it Annals of Global Analysis and Geometry}, {\it 57}: 1-13, 2020. 

\bibitem{michael} J. H. Michael and L. M. Simon, Sobolev and mean-value inequalities on generalized submanifolds of $R^n$, {\it Comm. Pure Appl. Math.} {\bf 26}, 361–379, 1973.

\bibitem{peterthesis} P.O. Olanipekun, Study of a four dimensional Willmore energy, {\it PhD Thesis}, Monash University, Melbourne, Australia, 2021, arXiv:2210.05924

\bibitem{reilly} R.C. Reilly,  Extrinsic rigidity theorems for compact submanifolds of the sphere. {\it J. Differential Geom.}
{\bf 4}, 487-497, 1970.

\bibitem{shu} S. Shu, Curvature and rigidity of Willmore submanifolds, {\it Tsukuba J. Math.}, {\bf 31}(1): 175-196, 2007. 

\bibitem{jsimons} J. Simons,  Minimal varieties in Riemannian manifolds, {\it Ann. of Math.}{\bf 88}(2): 62-105, 1968.

\bibitem{glen} G. Wheeler, Gap phenomenon for a class of fourth-order geometric differential operators on surfaces with boundary, {\it Proc. Amer. Math. Soc. } {\bf 143}(4): 1719-1737, 2015. 


\bibitem{zhang} Y. Zhang, Graham-Witten’s conformal invariant for closed four dimensional
submanifolds, {\it J. Math. Study} {\bf 54}(2): 200-226, 2021.














\end{thebibliography}
\end{document}